\documentclass[12pt,reqno,a4paper]{amsart}
\usepackage{euscript,fullpage,amsmath,amssymb}
\usepackage[T1]{fontenc}
\usepackage{enumerate}
\usepackage{yfonts}
\usepackage{graphicx}
\usepackage{graphicx,amssymb,amsmath,amsthm}
\usepackage{tikz}
\usepackage[pagewise]{lineno}%
\usetikzlibrary{patterns}
\usepackage{comment}
\usepackage{yfonts}

\newtheorem{thm}{Theorem}[section]
\newtheorem{lemma}[thm]{Lemma}
\newtheorem{remark}[thm]{Remark}

\newtheorem{proposition}[thm]{Proposition}
\newtheorem{example}[thm]{Example}

\providecommand{\floor}[1]{\left\lfloor \, #1 \, \right\rfloor}

\newcommand{\bn}{B_n^c}

\newcommand{\mcl}{\mathcal L}

\newcommand{\eps}{\epsilon}

\title{Extreme Value Theory with Spectral Techniques: application to a simple attractor.}
\date{\today}
\begin{document}
\author{Jason Atnip}
\thanks{J Atnip, School of Mathematics and Statistics, University of New South Wales, Sydney, NSW
2052, Australia.
E-mail: {\tt{j.atnip@unsw.edu.au}}}
\author{Nicolai Haydn}
\thanks{N Haydn, Mathematics Department, USC,
Los Angeles, 90089-2532. E-mail: {\tt {nhaydn@usc.edu}}}
 \author{Sandro Vaienti}
 \thanks{S Vaienti, Aix Marseille Universit\'e, Universit\'e de Toulon, CNRS, CPT, UMR 7332, Marseille, France.
 E-mail: {\tt {vaienti@cpt.univ-mrs.fr}}}

\date{\today}
\begin{abstract}
We give a  brief account of application of extreme value theory in dynamical systems by using perturbation techniques associated to   the transfer operator. We will apply it to the baker's map and we will get a precise formula for the extremal index. We will also show that the statistics of the number of visits in small sets is compound Poisson distributed.
\end{abstract}


\maketitle


\tableofcontents

\section{Introduction}
Extreme value theory (EVT) has been widely studied in the last years in application to dynamical systems both deterministic and random. A review of the recent results with an exhaustive bibliography is given in our collective work \cite{book}. As we will see, there is a close connection between EVT and the statistics of recurrence and both could be worked out simultaneously by using perturbations theories of the transfer operator. This powerful approach is limited to systems with quasi-compact transfer operators and exponential decay of correlations; nevertheless it can be applied to situations where more standard techniques meet obstructions and difficulties, in particular to:\\
- non-stationary  and random dynamical systems,\\
- observable with non-trivial extremal sets,\\
- higher-dimensional systems.\\
Another big advantage of this technique is the possibility of defining in a precise and universal way the extremal index (EI). We defer to our recent paper \cite{v2} for a critical discussion of this issue with several explicit computations of the EI in new situations. The germ of the perturbative technique of the transfer operator applied to EVT is in the fundamental paper \cite{KL} by G. Keller and C. Liverani; the explicit connection with recurrence and extreme value theory has been done by G. Keller in the article \cite{GK}, which contains also a list of suggestions for further investigations. We successively applied this method to i.i.d.\ random transformations in \cite{v5, v2}, to randomly quenched dynamical systems in \cite{quenched},  to coupled maps on finite lattices in \cite{v3}, and  to open systems with targets and holes in  \cite{v4}.

 The object of this note is to illustrate this technique by  presenting a new application to a bi-dimensional invertible system. We will see that the perturbative technique could be applied in this case as well  provided one could find the good functional spaces where the transfer operator exhibits quasi-compactness.

 We will find a few limitations to a complete application of the theory and to its generalization to a wider class of  maps in higher dimensions, see Remarks \ref{RRR1} and \ref{RRR2}.\\

 When the first version of this paper circulated, the spectral technique discussed above did not allow us to get another property related to limiting return and hitting  times
distribution in small sets (sometimes also called {\em target} or {\em holes}), namely the  statistics of the number of visits, which takes usually the form of a compound Poisson distribution. This has been recently achieved in the paper  \cite{compound}, and it could be  easily applied to the system under investigation in this paper. We will briefly quote this technique in section \ref{PPP}. As for the EVT, such a technique suffers of the limitation imposed by the shape of the target sets and for choice of the parameters, see remark \ref{RRR2}. The former will be particularly important for us when we decide to use rectangular target set, see section \ref{fc}.  These case could be worked out with another technique  developed  by two of us, see \cite{HHVV}, which allows to recover compound Poisson distributions  for invertible maps in higher dimension and arbitrary small sets.
By using this  approach, we will be also able to construct an example for the {\em fat} baker map with a compound Poisson distribution which  is neither  the standard Poisson  nor  the P\`olya-Aeppli,which are the most common compound distributions.

We will finally discuss the extension to compound Poisson point process on the real line.

\section{A simple  example: the generalized baker's map}

\subsection{The map} We now treat an example for which there are not apparently established results for the  extreme value distributions. This example, the generalized baker's map, from now on simply abbreviated as baker's map,  is a prototype for uniformly hyperbolic transformations in more than one dimension, two in our case, and in order to study it with the transfer operator, we will introduce suitable anisotropic Banach spaces. Our original goal was to investigate directly larger classes of uniformly hyperbolic maps, including Anosov ones, but, as we said above, the generalizations do not seem straightforward; we will explain the reason later on. With the usual probabilistic approaches extreme value distributions have been obtained for the linear automorphisms of the torus in \cite{NF}. \\

We will refer to the baker's transformation  studied in Section 2.1 in \cite{DDHH}, but we will write it in a particular case in order to make the exposition more accessible. The baker's transformation $T(x_n, y_n)$ is defined on the unit square $X=[0,1]^2 \subset \mathbb{R}^2$ into itself by:
\begin{eqnarray}\label{bibi}
x_{n+1} =
\bigg \{
\begin{array}{rl}
\gamma_a x_n & \text{if} \ y_n<\alpha \\
(1-\gamma_b)+\gamma_bx_n & \text{if} \ y_n>\alpha\\
\end{array}\\
y_{n+1} =
\bigg \{
\begin{array}{rl}
\frac{1}{\alpha}y_n & \text{if} \ y_n<\alpha \\
\frac{1}{\upsilon}(y_n-\alpha) & \text{if} \ y_n>\alpha,\\
\end{array}
\end{eqnarray}
with $\upsilon=1-\alpha,$ $\gamma_a+\gamma_b\le 1,$ see Fig. 1. To simplify some of the  next formulae, we will take $\alpha=\upsilon= 0.5$ and $\gamma_a=\gamma_b< 0.5.$ This last value must be strictly less than $1/2$ since Lemma \ref{ldc} requires the stable dimension $d_s$ strictly less than one, which corresponds to a fractal invariant set ({\em thin baker's map}). This condition will be relaxed in the example \ref{IIEE} ({\em fat baker's map}), but using an approach different of the spectral one leading to Lemma \ref{ldc}.
\begin{figure}
\begin{tikzpicture}
\draw [pattern=north west lines, pattern color=green] (0,0) rectangle (6,2);
\draw [pattern=north east lines, pattern color=black] (0,2) rectangle (6,6);
\draw [pattern=north west lines, pattern color=green] (10,0) rectangle (12.5,6);
\draw [pattern=north east lines, pattern color=black] (14.5,0) rectangle (16,6);
\draw[-](0,0) --  (0,6);
\draw[-](6,0) --  (6,6);
\draw (0,0) node[below left]{$0$};
\draw (0,6) node[left]{$1$} ;
\draw (-0.5,3) node[left]{$y$} ;
\draw (0,2) node[left]{$\alpha$} ;
\draw (3,-0.5) node[below]{$x$} ;
\draw (6,0) node[below]{$1$} ;
\draw (10,0) node[below left]{$0$} ;
\draw (16,0) node[below]{$1$} ;
\draw (10,6) node[left]{$1$} ;
\draw (9.5,3) node[left]{$y$} ;
\draw (13,-0.5) node[below]{$x$} ;
\draw (12.5,0) node[below]{$\gamma_a$} ;
\draw (14.5,0) node[below]{$1-\gamma_b$} ;
\draw[-] (10,0) -- (16,0);
\draw[-] (10,6) -- (16,6);
\draw[->,>=stealth] (6,4) -- (14.5,3);
\draw[->,>=stealth] (6,1) -- (10,3);

\end{tikzpicture}
\caption{Action of the baker's map on the unit square. The lower part of the square is mapped to the left part and the upper part is mapped to the right part.}
\label{tik}
\end{figure}
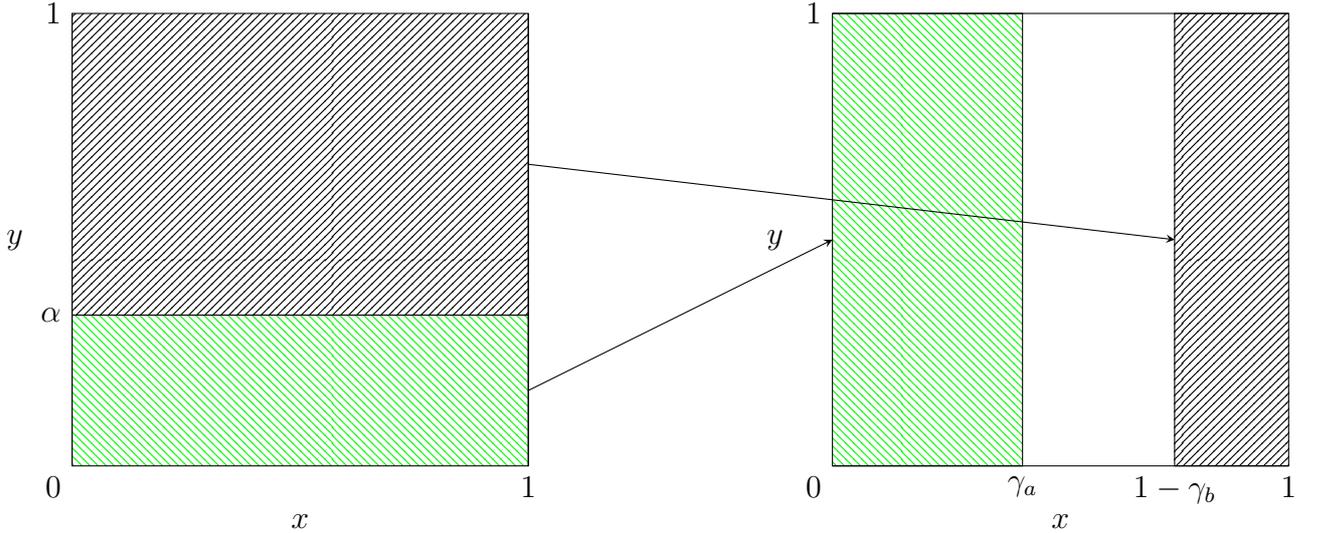

The map $T$ is discontinuous at  the horizontal line $\Gamma:\{y=\alpha\}.$ The singularity curves for $T^{l}, l>1$ are given by $T^{-l}\Gamma$ and they are constructed in this way:  take the preimages $T_Y^{-l}(\alpha)$ of $y=\alpha$ on the $y$-axis according to the map:
\begin{equation}\label{rty}
T_Y(y)=
\bigg \{
\begin{array}{rl}
 \frac{1}{\alpha} y, y<\alpha\\
\frac{1}{\upsilon}y-\frac{\alpha}{\upsilon}, y\ge\alpha.\\
\end{array}
\end{equation}
Then $T^{-l}\Gamma=\{y=T_Y^{-l}(\alpha)\}.$  Any other horizontal line will be a stable manifold of $T.$
The invariant non-wandering set $\Lambda$ will be at the end an attractor foliated by vertical lines which are all unstable manifolds. We denote by $\mathcal{W}^s(\mathcal{W}^u)$ the set of full horizontal (vertical) stable (unstable) manifolds of length  $1$ just constructed. We point out that a stable horizontal manifold $W_s$ will originate two  disjoint full stable manifold when iterated backward by $T^{-1},$ not for the presence of singularity, but because the map $T^{-1}$ will only be defined on the two images of $T(X)$ as illustrated in Fig. 1.
\subsection{The functional space}
In order to obtain useful spectral information from the transfer operator $\mathcal{L}$, its action is restricted to a Banach space $\mathcal{B}$.
We now give the construction of the norms on $\mathcal{B}$ and an associated ``weak'' space $\mathcal{B}_w$ in the case of the baker's map, following partly the exposition in \cite{DDHH}. In this case the norms will be constructed  directly on the horizontal stable manifolds instead of admissible leaves, which
 are smooth curves in approximately the stable direction, see \cite{DL}. As we anticipated above, we follow \cite{DDHH}, but we slightly change the definition of the stable norms by adapting ourselves to that originally introduced in \cite{DL}. Let us explain why.
 First of all we will consider the collection  $\Sigma$
  of all the  intervals $W$ of length less or equal to $1$ that are contained in the same stable manifold 
  $W_s\in \mathcal{W}^s$. We will take such a value equal to $\gamma_a$ for reasons which will be clear in the next considerations. Instead in \cite{DL}, $\Sigma$ was the set of
full horizontal line segments of length $1$ in $X$. The reason for our choice is that we will introduce small target sets $B_n$  and the preimages of such sets will cut the $W_s.$ The smaller pieces generated in this way will enter the three norms given below and therefore it will be useful to count such pieces in $\Sigma.$ \\
 
 We begin to consider  the set $C^{\varrho}$ of continuous complex-valued functions over $X$  with H\"older exponent $0\le \varrho\le 1.$ When we set $C^1$ we mean  $C^\varrho$ with $\varrho=1,$ which is simply Lipschitz. Given a stable leaf $W$ and a H\"older function  $\varphi,$ we  define the norm along $W$ as:
$$
 |\varphi \rvert_{C^{\varrho}(W)}=|\varphi |_{C^0(W)}+H^{\varrho}(\varphi), \  H^{\varrho}(\varphi)=\sup_{\substack{x,y\in W \\x\neq y}}\frac{|\varphi(x)-\varphi(y)|}{|x-y|^{\varrho}}.
$$
 
Another norm will be considered later on, namely
 \begin{equation}\label{TF}
 \lvert \varphi \rvert_{W, \varrho}:=\lvert W\rvert^\varrho \cdot \lvert \varphi \rvert_{C^{\varrho}(W)},
\end{equation}
where $|W|$ denotes the length of $W.$

We now follow closely section 2.2 in \cite{DDHH} and put 
\begin{equation}\label{lll}
|\varphi|_{C^{\varrho}(\mathcal{W}^s)}=\sup_{W\in \mathcal{W}^s}|\varphi \rvert_{C^{\varrho}(W)}.
\end{equation}
We call $C^1(\mathcal{W}^s)$ the set of functions that are Lipschitz along stable manifolds, i.e. for which the quantity (\ref{lll}) is finite. For $\varrho<1,$ we set $C^{\varrho}(\mathcal{W}^s)$ the completion of $C^1(\mathcal{W}^s)$ in the $|\cdot|_{C^{\varrho}(\mathcal{W}^s)}$ norm. Analogously, $C^{\varrho}(W^s)$ denotes the completion of $C^1(W^s)$ in the $|\cdot|_{C^{\varrho}(W^s)}$ norm. One can show that $|\varphi|_{C^{\varrho}(\mathcal{W}^s)}$ is a Banach space for $0\le \varrho\le 1;$ the Banach space $|\varphi|_{C^{\varrho}(\mathcal{W}^u)}$ is defined similarly. By the very structure of the map it follows that whenever $\varphi\in C^{\varrho}(\mathcal{W}^s),$ then $\varphi\circ T\in C^{\varrho}(\mathcal{W}^s).$ This allows to define the transfer operator $\mcl$ associated with  $T$ on the dual space $(C^{\varrho}(\mathcal{W}^s))^*$ as\footnote{Notice that although the map $T$ is discontinuous, the fact that  $\varphi\in C^{\varrho}(\mathcal{W}^s)$ implies $\varphi\circ T\in C^{\varrho}(\mathcal{W}^s),$ allows us to define the transfer operator on the space $(C^{\varrho}(\mathcal{W}^s))^*$ without the need for two scales of space as in \cite{DL}. }
$$
  (\mcl h)(\varphi)=h(\varphi \circ T), \ \forall{\varphi}\in C^{\varrho}(\mathcal{W}^s), \ h\in (C^{\varrho}(\mathcal{W}^s)^*.
$$
If we denote with $m_L$ the Lebesgue measure over $X$ and we take $h\in C^{1}(\mathcal{W}^u),$ then we identify $h$ with the measure $hdm_L$ so that $h\in (C^{\varrho}(\mathcal{W}^s))^*$ and $\mcl h$ is  now identified with the measure having density
\begin{equation}\label{to}
 \mcl h(x)=\bigg{(}\frac{h} {\lvert \det DT\rvert}\bigg{)}\circ T^{-1}(x)=\frac{h\circ T^{-1}(x)}{\alpha^{-1}\gamma_a},
\end{equation}
where the last equality on the r.h.s.\ uses the particular choices for the parameters defining the map $T.$
When $h\in C^{1}(\mathcal{W}^u),$ we therefore set 
\begin{equation}\label{d}
h(\varphi)=\int_X h\varphi \,dm_L,  \quad \text{for $\varphi \in C^{1}(\mathcal{W}^s)$.}
\end{equation}
We are now ready to construct the Banach spaces.

For $h\in C^{1}(\mathcal{W}^u)$ we define the {\em weak norm} of $h$ by
\[
 \lvert h\rvert_w=\sup_{W\in \Sigma}\sup_{\substack{\varphi \in C^1(W) \\ \lvert \varphi \rvert_{C^1(W)}\le 1}}\bigg{\lvert} \int_W h\varphi \, dm \bigg{\rvert}
\]
where $dm$ is the unnormalized Lebesgue measure along $W.$ 

We now   take\footnote{The bound $\beta\le \min(\kappa, 1-\kappa)$ is needed in the proof of Lemma 3.1 in \cite{DDHH}. Notice that such a lemma only require $\beta\le 1-\kappa.$ The additional constraint $\beta\le \kappa$ comes from the fact that in the proof of Lemma 3.1 in \cite{DDHH}, in particular in the estimate of the strong unstable norm, there are not {\em unmatched} pieces since all the stable leaves have full length, see also footnote 3.} $(\kappa,\beta)\in (0,1)$ with  $0<\beta\le \min(\kappa, 1-\kappa).$

The strong stable norm is defined as:
\begin{equation}
\label{ssnorm}
 \lVert h\rVert_s=\sup_{W\in \Sigma}\sup_{\substack{\varphi \in C^{\kappa}(W) \\ \lvert \varphi \rvert_{W, \kappa}\le 1}}\bigg{\lvert} \int_W h\varphi \, dm \bigg{\rvert}.
\end{equation}
We then need to define the strong unstable norm which allows us to compare expectations along different stable manifolds.
If $W_1$ is a subset of the stable manifold $W_s$ we could parameterize it as $( s_{W_1}, t)$ where $s_{W_1}$ is the common ordinate of the points in $W_1$ and $t\in [a_1,b_1]\subset [0,1].$ If $W_2$ is a subset of another stable manifold, parametrized as $(s_{W_2}, t)$ with $t\in[a_2,b_2],$ we pose
$$
d(W_1, W_2)=|s_{W_1}-s_{W_2}|+|[a_1,b_1]\Delta[a_2,b_2]|,
$$
where $\Delta$ means the symmetric difference, and for test functions $\varphi_i\in C^1(W_i), i=1,2:$
$$
d_{\kappa}(\varphi_1, \varphi_2)=\sup_{t\in [a_1, b_1]\cap[a_2,b_2]}|\varphi_1(s_{W_1},t)-\varphi_2(s_{W_2},t)|_{C^{\kappa}(W_i)}.
$$

 The strong unstable norm of $h$ is defined as

\begin{equation}
\label{sunorm}
 \lVert h\rVert_u=\sup_{\epsilon\le 1}\sup_{\substack{W_1, W_2\in \mathcal{W}_s \\ d(W_1, W_2)\le \epsilon}}\sup_{\substack {\varphi_i\in C^1(W_i) \\ \lvert \varphi_i \rvert_{C^1(W)}\le 1\\ d_{\kappa}(\varphi_1, \varphi_2)\le \epsilon}}\frac{1}{\epsilon^{\beta}}\left|
 \int_{W_1}h\varphi_1 dm- \int_{W_2}h\varphi_2 dm\right|,
\end{equation}
Finally we can define the {\em strong norm} of $h$ by \[\lVert h\rVert=\lVert h\rVert_s+b\lVert h\rVert_u,\]
where $b$ is a small constant to be fixed later on.\\

We set $\mathcal B$  to be the completion of $C^{1}(\mathcal{W}^u)$ with respect to the norm $\lVert \cdot \rVert$, and,
 similarly, we define
$\mathcal B_w$ to be the completion of $C^{1}(\mathcal{W}^u)$ with respect to the norm $\lvert \cdot \rvert_w$.\\

We now list a few important results whose proof can be found in \cite{DDHH} and which we will use frequently in the next sections.
\begin{itemize}
\item (Lemma 2.4, \cite{DDHH}) For any $\beta'\in (\beta, 1),$ we have the following sequence of continuous embeddings,
$$
C^1(X)\hookrightarrow C^{\beta'}(\mathcal{W}^u)\hookrightarrow \mathcal{B}\hookrightarrow \mathcal{B}_w\hookrightarrow (C^{1}(\mathcal{W}^s))^*.
$$
Moreover the embedding $\mathcal{B}\hookrightarrow \mathcal{B}_w$ is relatively compact.
\item For $h\in \mathcal{B}$ and $\varphi\in C^1(\mathcal{W}^s)$ we have
\begin{equation}\label{pif}
|h(\varphi)|\le |h|_w |\varphi|_{C^1(\mathcal{W}^s)}.
\end{equation}
Moreover
\begin{equation}\label{ggg}
|h|_w\le \|h\|_s.
\end{equation}
\item (Lemma 4.1, \cite{DDHH}). The transfer operator $\mcl$ is a bounded, linear operator on both $\mathcal{B}$ and $\mathcal{B}_w.$
\item (Lemma 3.1, \cite{DDHH}). If $g\in C^1(X)$ and $h\in\mathcal{B},$ then\footnote{In \cite{DDHH} the constant on the r.h.s.of (\ref{cico}) is simply $3.$ We should modify it since the presence of unmatched pieces adds two more contributions of  the factor $|g|_{C1(X)}\|h\|_s$ in the computation of the strong unstable norm. Finally the factor $b$ comes from the very definition of the Banach norm $\|\cdot\|.$}
\begin{equation}\label{cico}
\|gh\|\le (5b+1) |g|_{C^1(X)}\|h\|.
\end{equation}
\item (Theorem 2.5, \cite{DDHH}). $\mcl$ is quasi-compact as an operator on $\mathcal{B}$. Its spectral radius is $1$ and its essential spectral radius is bounded by $\max\{ \lambda_a^{\kappa}, \alpha^{\beta}\}<1.$

Then:\\
(a) $\mcl$ has $1$ as a simple eigenvalue and all other eigenvalues have modulus less than $1.$\\
(b) There is a unique solution $\mu \in \mathcal{B}$ of $\mcl \mu = \mu,$ with $\mu(1)=1$ and such a solution is the Sinai-Bowen-Ruelle,  SRB-measure. Its conditional measures on unstable leaves are equal to arclength.\\
(c) There exists $\textfrak{a}<1$ and $\textfrak{C}>0$ such that for any $h\in \mathcal{B}$ with $h(1)=1,$ we have
$$
\|\mcl^n h-\mu\|\le \textfrak{C} \textfrak{a}^n\|h\|, \ \forall n\ge 0.
$$

    \end{itemize}



\section{The spectral approach for EVT}
\subsection{Formulation of the problem}\label{fop}
We now take a ball $B(z, r)$ of center $z\in X$ and radius $r$ and denote with $B(z, r)^c$ its complement, where  $d(\cdot, \cdot)$  is the Euclidean metric.

Let us consider for $x\in X$ the observable\footnote{See section \ref{dse} for a discussion about the choice of the observable.}
\begin{equation}\label{OO}
\Xi(x)=-\log d(x,z)
\end{equation}
and the function
\begin{equation}\label{IP}
M_n (x):=\max\{\Xi(x), \cdots, \Xi(T^{n-1}x)\}.
\end{equation}
For $ u \in \mathbb{R}_+$, we are interested in the distribution of $M_n\le u,$ where $M_n$ is now seen as a random variable on the  probability space $(X, \mu).$  Notice that the event $\{M_n\le u \}$ is equivalent to the set  $\{ \Xi\le u, \ldots, \Xi\circ T^{n-1}\le u\}$ which in turn coincides with  the set
 $$
 E_n:=B(z,e^{-u})^c\cap  T^{-1}B(z,e^{-u})^c\cap\cdots \cap  T^{-(n-1)} B(z,e^{-u})^c.
 $$
  We are therefore following points which will enter the ball $B(z,e^{-u})$  for the first time after at least $n$ steps (see e.g.  eq. (\ref{retour}) in sect. \ref{dse}), and $u\mapsto \mu(E_n)$ is the distribution function of the maximum of the observable $\Xi\circ T^j, j=0,\dots, n-1.$ It is well known from elementary probability that the distribution of the maximum of a  sequence of i.i.d.\ random variables is degenerate. One way to overcome this is to make the {\em boundary level} $u$ depend upon the time $n$ in such a way the sequence $u_n$ grows to infinity and gives, hopefully,  a non-degenerate limit for $\mu(M_n\le u_n).$\\

From now on we set: $B_n=B(z, e^{-u_n})$ and $B_n^c$ the complement of $B_n;$ the dependence upon the "center" $z$ will be discussed in Remark \ref{RRR1}.\\
 We easily have
\begin{equation}\label{RI}
\mu (M_n\le u_n)=\int {\bf 1}_{B_{n}^c}(x) {\bf 1}_{B_{n}^c}(Tx)\cdots {\bf 1}_{B_{n}^c}(T^{n-1}x) \,d\mu.
\end{equation}
 By introducing the perturbed operator, for $h\in \mathcal B$:
\begin{equation}\label{PO}
\mcl_{n}h:=\mcl({\bf 1}_{B_{n}^c} h),
\end{equation}
we can write (\ref{RI}) as
\begin{equation}\label{RI2}
\mu (M_n\le u_n)= \mcl_{n}^n\mu(1).
\end{equation}
Notice that
$$
\mcl_n^kh=\mcl(h{\bf 1}_{B_{n}^c}{\bf 1}_{B_{n}^c}\circ T\dots{\bf 1}_{B_{n}^c}\circ T^{k-1}).
$$
\subsection{Target sets and the space $\mathcal{B}$}
We explicitly used above  two facts which require justification. 
\subsubsection {${\bf 1}_{B_{n}^c}$\label{subsub1} is in the Banach space $\mathcal{B}$} Of course the same proof should  hold for functions of  the form  ${\bf 1}_{B_{n}^c}{\bf 1}_{B_{n}^c}\circ T\dots{\bf 1}_{B_{n}^c}\circ T^{k-1}.$ 
The geometric shape of the sets $B_{n}^c\cap T^{-1}B_{n}^c \cap  \cdots \cap T^{-(k-1)}B_{n}^c$
plays an important role in the proof. Those sets are equivalently given by $\left(B_{n}\cup T^{-1}B_{n} \cup  \cdots \cup T^{-(k-1)}B_{n}\right)^c$ and we call $B^{(k)}$ one of them.
 Suppose we could find a sequence $\{h_l\}_{l\in \mathbb{N}}$ in $C^1(X)$ which is Cauchy in $\mathcal{B}$ and such that for any  $\varphi\in C^1(\mathcal{W}^s) $ we have
\begin{equation}\label{lio}
\int h_l \varphi dm\rightarrow \int{\bf 1}_{B^{(k)}}\varphi dm, \ l\rightarrow \infty.
\end{equation}
Then ${\bf 1}_{B^{(k)}}$ is in $\mathcal{B}$ since the latter  is continuously embedded in the dual space $(C^1(\mathcal{W}^s))^*,$ see Lemma 2.4 in \cite{DDHH} quoted in section 2. We now construct such a sequence. Take a bump function $\phi$ with support in the unit ball of $\mathbb{R}^2,$ normalized to $1$ and put $\phi_{\delta}(x):=\frac{1}{\delta^2}\phi(\frac{x}{\delta}),$ where $\delta$ designs the $\delta$-neighborhood $B^{(k)}_{\delta}:=\{x\in \mathbb{R}^2; \text{dist}(x, B^{(k)})\le \delta\}$ of $B^{(k)}.$
Then define the convolution product
\begin{equation}\label{convoy}
h_{\delta}:= {\bf 1}_{B^{(k)}_{\delta}}*{\phi_{\delta}}.
\end{equation}
Then  $h_{\delta}$ is equal to $1$ on $B^{(k)}$ and equal to zero outside the $2\delta$-neighborhood $B^{(k)}_{2\delta}.$ Moreover it is straightforward to get (\ref{lio}). It remains to prove that $\{h_{\delta}\}_{\delta>0}$ is Cauchy, for $\delta\rightarrow 0.$ Call $U_{\delta}=B^{(k)}_{2\delta}\setminus B^{(k)}.$ To control the strong stable norm, we observe that, if $\delta_2<\delta_1$
\begin{equation}\label{schif}
\left|\int_W(h_{\delta_1}-h_{\delta_2})\varphi dm\right|=\left|\int_{W\cap U_{\delta_1}}(h_{\delta_1}-h_{\delta_2})\varphi dm\right|\le 2 |W|^{-\kappa}|W\cap U_{\delta_1}|\le 2 |W\cap U_{\delta_1}|^{1-\kappa}
\end{equation}
There are  now two cases:\\
(i) suppose first that $|W|\le \delta_1;$ then (\ref{schif})$\le \delta_1^{1-\kappa}.$\\
(ii) suppose now that $|W|>\delta_1.$ As we will write in footnote $5,$ each $W$ could meet at most $k-1$  sets  of the form $T^{-j}B_n, j=1,\dots,k.$ These sets are ellipses with the major axis along the stable manifolds. Therefore each $W$ could meet at most $(k-1)$ $\delta_1-$neighborhoods of the preimages $T^{-j}B_n, j=1,\dots,k.$ It is a simple exercise to show that the maximum intersection of $W$ with one of the previous  $\delta_1-$neighborhoods is bounded by a constant $\tilde{C}$ depending only on the size of $X$ times $(\delta_1)^{1/2}.$ Then (\ref{schif})$\le 2(k-1)[\tilde{C}\delta_1^{1/2}]^{1-\kappa}$. In conclusion
$$
\|h_{\delta_1}-h_{\delta_2}\|_s\le  2(k-1)[\tilde{C}\delta_1^{1/2}]^{1-\kappa}.
$$
We now compute the strong unstable norm. We proceed in two different manners. 
First of all we could  simply bound the difference
\begin{equation}\label{fin}
\frac{1}{\epsilon^{\beta}}\left|\int_{W_1}(h_{\delta_1}-h_{\delta_2})\varphi_1 dm-\int_{W_2}(h_{\delta_1}-h_{\delta_2})\varphi_2 dm\right|
\end{equation}
by $\frac{1}{\epsilon^{\beta}}4(k-1)\tilde{C}\delta_1^{1/2},$  since the term $h_{\delta_1}-h_{\delta_2}$ is different from zero only on the intersections of the manifolds $W_1, W_2$ with a $\delta_1$-neighborhood. 
We now pass to a finer estimate of (\ref{fin}) and we will use the same trick to control the strong unstable norm in the Lasota-Yorke inequality, see section \ref{ese}. We will see that giving two stable manifolds $W_1, W_2$ at a distance at most $\epsilon,$ there will be two {\em matched} subsets of those manifolds whose points have the same $y$-ordinate, and the $x$-components belong to the same interval. The complement of the matched piece on each manifold has length less or equal to $\epsilon$ ({\em unmatched pieces}). 
The contribution given by those  two unmatched pieces  is
$
2(k-1)\epsilon^{1-\beta}.
$
We now parametrize the two matched pieces, where all the points of $W_1$ (resp. $W_2$), have the same ordinate $s_1$ (resp. $s_2$) and the abscissa $t$ varies in the interval $I_{1,2}.$ Then we can write the difference of the two integrals in (\ref{fin}) as
\begin{equation}\label{io}
\left|\int_{I_{1,2}}(h_{\delta_1}-h_{\delta_2})(s_1,t)\varphi_1(s_1,t)dt-\int_{I_{1,2}}(h_{\delta_1}-h_{\delta_2})(s_2,t)\varphi_2(s_2,t)dt\right|.
\end{equation}
Notice that from now on  it only matters the intersection of $W_1, W_2$ with $U_{\delta_1},$ since outside it the quantity (\ref{io}) is zero.
We begin to control the piece $D_1:=\int_{I_{1,2}}h_{\delta_1}(s_1,t)\varphi_1(s_1,t)dt-\int_{I_{1,2}}h_{\delta_1}(s_2,t)\varphi_2(s_2,t)dt,$ the other one involving $h_{\delta_2},$ call it $D_2,$ being the same. We split it as
$$
D_1=\int_{I_{1,2}} h_{\delta_1}(s_1,t)\varphi_1(s_1,t)dt-\int_{I_{1,2}}h_{\delta_1}(s_1,t)\tilde{\varphi}_1(s_1,t)dt+$$
$$
\int_{I_{1,2}}h_{\delta_1}(s_1,t)\tilde{\varphi}_1(s_1,t)dt-\int_{I_{1,2}}h_{\delta_1}(s_2,t)\varphi_2(s_2,t)dt,
$$
where we put
$$
\tilde{\varphi}_1(s_1,t)=\varphi_2(s_2,t), \ t\in I_{1,2}.
$$
The absolute value of the first difference in $D_1$ is simply bounded by $\epsilon$ (remember $d_{\kappa}(\varphi_1, \varphi_2)\le \epsilon$), times $2(k-1)\tilde{C}\delta_1^{1/2}.$ The absolute value of the second piece in $D_1$ is bounded by $
2(k-1)\tilde{C}\delta_1^{1/2}$ times $H(h_{\delta_1})\epsilon$,
where $H(h_{\delta_1})\le \frac{C_{\phi}}{\delta^3_1}$ is the Lipschitz constant of $h_{\delta_1}$ and $C_{\phi}$ depends only on $\phi.$ By dividing for $\epsilon^{\beta}$ we have that $|D_1|+|D_2|$ is bounded by 
\begin{equation}\label{refi}
2(k-1)\frac{\hat{C}}{\delta_1^{5/2}}\epsilon^{1-\beta},
\end{equation}
where $\hat{C}=\max(\tilde{C}, C_{\phi}).$ In conclusion we have (the term coming from the unmatched pieces being incorporated  in (\ref{refi}))
$$
\|h_{\delta_1}-h_{\delta_1}\|_u\le \min\{4(k-1)\frac{\hat{C}}{\delta_1^{5/2}}\epsilon^{1-\beta}, \frac{1}{\epsilon^{\beta}}4(k-1)\tilde{C}\delta_1^{1/2}\}.
$$
By interpolating we finally get
$$
\|h_{\delta_1}-h_{\delta_1}\|_u\le 4(k-1)\hat{C}\epsilon^{1-(q+\beta)}\delta_1^{3q-5/2},
$$
where $q\in (0,1)$ must be chosen such that $\beta+q<1$ and $q>5/6.$
\subsubsection {${\bf 1}_{B^{(k)}}h\in \mathcal{B}$}\label{subsub2} 

Take again a set like $B^{(k)}$ such that ${\bf 1}_{B^{(k)}}\in \mathcal{B};$ what we need is that ${\bf 1}_{B^{(k)}}h\in \mathcal{B},$ where $h\in \mathcal{B}.$ First of all we have to define the object ${\bf 1}_{B^{(k)}}h\in \mathcal{B}.$ Take a sequence $\{h_l\}_{l\ge 1}\in C^{1}(\mathcal{W}^u)$ converging to $h$ in the $\mathcal{B}-$norm. Whenever ${\bf 1}_{B^{(k)}}h_l\in \mathcal{B},$ we set
\begin{equation}\label{poi}
{\bf 1}_{B^{(k)}}h=\lim_{l\rightarrow \infty}{\bf 1}_{B^{(k)}}h_l,
\end{equation}
provided the limit exists. 
So, first of all we have to show that for any $l,$ ${\bf 1}_{B^{(k)}} h_l\in \mathcal{B}.$ This is proved exactly in the same manner as in the previous  item, where $h_l=1.$

We are left by showing that ${\bf 1}_{B^{(k)}} h_l$ is Cauchy. To get it we begin to prove a preliminary result, namely in the computation of the strong stable and unstable norm of ${\bf 1}_{B^{(k)}} f,$  where ${\bf 1}_{B^{(k)}} \in \mathcal{B}$ and $f\in C^{1}(\mathcal{W}^u),$  such norms can be computed by using directly the (non smooth) product ${\bf 1}_{B^{(k)}} f.$  
 By (\ref{poi}), if we call $g_l$ a sequence converging to ${\bf 1}_{B^{(k)}},$ we put
$$
{\bf 1}_{B^{(k)}} f=\lim_{l\rightarrow \infty}g_l f,
$$
but now we are sure the limit exists since the sequence $g_lf$ is Cauchy by (\ref{cico}):
$$
\|(g_l-g_k)f\|\le (5b+1) \ \|g_l-g_k\|\ |f|_{C^1(X)}.
$$
Therefore we have that 
$$
\|g_lf\|\rightarrow \|{\bf 1}_{B^{(k)}} f\|,
$$
and the norm on the l.h.s. is the norm "before" completion. 
So we have
$$
 \|{\bf 1}_{B^{(k)}} f\|=\lim_{l\rightarrow \infty}\left(\|g_lf\|_s+ b\|g_lf\|_u\right).
$$
The result follows by replacing the strong stable and unstable norms on the right hand side respectively with the expressions (\ref{ssnorm}) and (\ref{sunorm}) and by passing the limit inside the integrals by dominated convergence. 
The same argument shows also that for $h\in C^{1}(\mathcal{W}^u):$

\begin{equation}\label{sez}
{\bf 1}_{B^{(k)}}h(\varphi)={\bf 1}_{B^{(k)}}(h \varphi)=\int_X {\bf 1}_{B^{(k)}}h_l\varphi dm.
\end{equation}
We now return to prove that  ${\bf 1}_{B^{(k)}} h_l$ is Cauchy by computing the norm of the generic element ${\bf 1}_{B^{(k)}} h,$ $h\in C^{1}(\mathcal{W}^u),$ directly  along the stable manifolds and showing that it is bounded by a constant depending only on $B^{(k)}$ times $\|h\|.$\footnote{Look at Lemma 4.3 in \cite{DDMM} for a similar computation.}
 If we take a stable manifold $W$ of length at most $\gamma_a,$ the intersection $W\cap B^{(k)}$ is given by a finite number $\#(W,k)$ of smaller stable intervals $W_i, 1\le i\le \#(W,k).$ The latter are generated by removing from $W$ the intersections with the preimages of $B_n$ up to order $k,$ see the beginning of section \ref{subsub1}; therefore $\#(W,k)\le k, \forall W,$ as explained in footnote $7.$ By using the arguments in {\bf A2} or in {\bf A3} in the next sections, it is straightforward to check that  $||{\bf 1}_{B^{(k)}}h||_s\le k ||h||_s\;\; \text{and} \;\; |{\bf 1}_{B^{(k)}}h|_w\le k |h|_w.$ It remains to compute the strong unstable norm and this reduces to bound the difference, for a smooth $h$: $\frac{1}{\epsilon^{\beta}}\left|
 \int_{W_1\cap B^{(k)}}h\varphi_1 dm- \int_{W_2\cap B^{(k)}}h\varphi_2 dm\right|,$ where $W_1 W_2, \varphi_1, \varphi_2$ verify the constraints given in (\ref{sunorm}). We remind that the preimages of $B_n$ of order $l\le k$ are ellipsis whose axis along the vertical unstable direction has length at most  $\alpha^l.$ We now split the computation in two parts. Suppose first that $\epsilon\ge0.5 \alpha^k.$ Then a rough estimate gives 
 \begin{equation}\label{u1}
     \frac{1}{\epsilon^{\beta}}\left|
 \int_{W_1\cap B^{(k)}}h\varphi_1 dm- \int_{W_2\cap B^{(k)}}h\varphi_2 dm\right|\le 4 \frac{1}{\alpha^{k\beta}} \|h\|_s.
 \end{equation}
 Take now $\epsilon <0.5 \alpha^k.$ By following the strategy used in dealing with  the strong unstable norm in sections \ref{subsub1} and \ref{ese}, we split the difference above over unmatched and matched pieces. Suppose now that $\#(W_1,k)>\#(W_2,k);$ then there could be at most  $\#(W_1,k)$ matched pieces given a final contribution bounded by $k\|h\|_u.$ We now count the unmatched pieces. First,  there are the two  intervals of length $\le \epsilon$ at the extremities of $W_1, W_2.$ Notice that there could be at most $\#(W_2,k)-1$ preimages of $B_n$ which cut both $W_1$ and $W_2;$ therefore there will be at most $2[\#(W_2,k)-1]$ unmatched pieces generated at the intersection with the boundaries of such preimages and the length of each of  those unmatched pieces is bounded by a constant $C_1(B_n,k),$ depending solely on the radius of the ball $B_n$ and on $k,$ times $\sqrt{\epsilon},$ as we argued in section \ref{ese}. But now there will be also $[\#(W_1,k)]-[\#(W_2,k)]$ unmatched pieces given by preimages of $B_n$ which cut $W_1$ but not $W_2.$ The length of those pieces will be again bounded by a constant $C_2(B_n,k)$ times $\sqrt{\epsilon}.$ Hence, we get for $\epsilon<0.5 \alpha^k:$\footnote{We incorporate the exponent $\kappa$ directly into the constants $C_1(B_n, k), C_2(B_n, k).$}
 \begin{equation}\label{uio}
 \left|
 \int_{W_1\cap B^{(k)}}h\varphi_1 dm- \int_{W_2\cap B^{(k)}}h\varphi_2 dm\right|\le 
\epsilon^{\beta}k\|h\|_u+\left[2\epsilon^{\kappa}+[2C_1(B_n,k)+ C_2(B_n,k)]k\epsilon^{\kappa/2}\right]\|h\|_s.
 \end{equation}
 In conclusion we get for $2\beta<\kappa:$
 $$
 ||{\bf 1}_{B^{(k)}}h||_u\le \max\{4 \frac{1}{\alpha^{k\beta}} , 2+[1+2C_1(B_n,k)+ C_2(B_n,k)]k\}\|h\|_u,
 $$
which immediately implies the Cauchly property for the sequence ${\bf 1}_{B^{(k)}} h_l.$\\

We finally  specialize (\ref{sez}) when $h$ is a Borel measure $\tilde{\mu}\in \mathcal{B}.$ Suppose also that the $\tilde{\mu}-$measure of the boundary of $B^{(k)}$ is zero. This happens for instance if $\tilde{\mu}$ is the SRB-measure, which is our case. Then if $h_l\rightarrow \tilde{\mu}$ and for a test function $\varphi$ we have
$$
{\bf 1}_{B^{(k)}}\tilde{\mu}(\varphi)=\lim_{l\rightarrow \infty}{\bf 1}_{B^{(k)}}h_l(\varphi)=\lim_{l\rightarrow \infty}\int {\bf 1}_{B^{(k)}}h_l\varphi dm=\int {\bf 1}_{B^{(k)}} \varphi d\tilde{\mu}=\tilde{\mu}({\bf 1}_{B^{(k)}} \varphi).
$$
where the third equality follows from Portmanteau theorem. 
This fact will be explicitly used in equation (\ref{DELTA2}) below.


\subsection{The perturbative approach}
The quasi-compacity of the operator $\mcl$ stated in (Theorem 2.5, \cite{DDHH}) and quoted in section 2, implies that\footnote{If $\varphi$ is a test function, eq. (\ref{QC}) means that $(\mcl h)(\varphi)=Z(h)\mu(\varphi)+Q(h)(\varphi).$}
\begin{equation}\label{QC}
\mcl=\mu\otimes Z+Q,
\end{equation}
where again  $\mu=\mcl  \mu$ is the  SRB measure in $\mathcal{B}$ normalized in such a way that $\mu(1)=1$ and spanning the one-dimensional eigenspace corresponding to the eigenvalue $1;$  $Z$ is the generator of the one-dimensional eigenspace of $\mcl^*$ in the dual space $\mathcal B^*$  corresponding to the eigenvalue $1$ and normalized in such a way that $Z(\mu)=1;$ finally $Q$ is a linear operator on $\mathcal B$ with spectral radius $sp(Q)$ strictly less than one.
We now introduce the assumptions which allow us to apply the perturbative technique of Keller and Liverani \cite{KL}. They are split in two blocks: {\bf A0}, {\bf A2} and {\bf A3} are needed to get  the quasi-compact decomposition (\ref{QCE}),  which extends to the perturbed operators $\mcl_n$ the same decomposition for $\mcl$ required by {\bf A1}.  The assumptions {\bf A4} and {\bf A5} together with (\ref{QCE}) are finally needed to apply the perturbative technique in \cite{KL} we referred  to at the beginning of this section.
\begin{itemize}
\item {\bf A0} $\mathcal{B}$ is continuously embedded into $\mathcal{B}_w.$
\item {\bf A1} The unperturbed operator $\mcl$ is quasi-compact in the sense expressed by (\ref{QC}).
    \item {\bf A2} There are constants $0<\rho<1, D_1, D_2, D_3>0,  \rho<M,$ such that $\forall n$ sufficiently large, $\forall h\in \mathcal{B}$ and $\forall k\in \mathbb{N}$ we have
        \begin{eqnarray}
       |\mcl_n^k h|_w\le D_1 M^k |h|_w,\\
       ||\mcl_n^k h||\le D_2 \rho^k||h||+D_3M^k|h|_w.\label{LLYY}
        \end{eqnarray}
        This will be proved below.
        \item {\bf A3}  We can bound the weak norm of  $(\mcl-\mcl_{n})h,$ with $h\in \mathcal B,$ in terms of the norm of $h$ as:
            $$
            |(\mcl-\mcl_{n})h|_w\le \chi_n ||h||\,
            $$
            where $\chi_n$ is a sequence converging to zero.  We give immediately the proof of this fact since it is  achieved  by a simple adaptation of the computation of the strong stable norm in the proof of item {\bf A2} below. Looking in fact at the notations and at the steps of such a demonstration, we have  to control the term: $\int_W(\mcl-\mcl_{n})h\varphi dm=\int_W \mcl({\bf 1}_{B_n} h)\varphi dm =\sum_{i=1,2}\int_{W_{i}\cap B_n} h(y) \varphi(Ty)\alpha\, dm(y)\le \|h\|_s |B_n|^{\kappa}.$  Then  $\chi_n=|B_n|^{\kappa}.$\\
           
            Thanks to the  assumptions {\bf A2} ({\em uniform Lasota-Yorke inequalities}) and {\bf A3} ({\em closeness of the operators in the triple norm}), we can apply the spectral  theory in \cite{KL2},\footnote{This spectral theory also requires that if $z$ is in the spectrum of $\mcl_n$ and $|z|>s,$ then $z$ is not in the residual spectrum of $\mcl_n.$ This last fact is guaranteed by {\bf A0} which implies that the spectral radius of $\mcl_n$ is bounded by $s.$ } and get  that the decomposition (\ref{QC}) holds for $n$ large enough, namely
\begin{align}\label{QCE}
\lambda_{n}^{-1}\mcl_{n}=\mu_{n}\otimes Z_{n}+Q_{n}, \\
 \mcl_{n}\mu_{n}=\lambda_{n}\mu_{n},\\
 Z_{n}\mcl_{n}=\lambda_{n}Z_{n},\\
 Q_{n}(\mu_{n})=0,\; \ Z_{n}Q_{n}=0, \label{QCF}
\end{align}
where $\lambda_{n}\in \mathbb{C},$ $ \mu_{n}\in \mathcal B,$ $   Z_{n}\in \mathcal B^*,$ $  Q_{n}\in \mathcal B,$ and $\sup_n sp(Q_n)\le sp(Q).$
We observe that the previous assumptions (\ref{QCE})--(\ref{QCF}) imply that $Z_{n} (\mu_{n})=1, \forall n;$ moreover $\mu_{n}$ can be normalized in such a way that $\mu_n(1)=1$ and $Z(\mu_n)=1,$ see \cite{KL}.\\
           
            We now state  assumption {\bf A4} leaving  {\bf A5} to section \ref{fc}.  
            \item {\bf A4} If we define
            \begin{equation}\label{DELTA}
\Delta_{n}=Z(\mcl-\mcl_{n})(\mu),
\end{equation}
and for $h\in \mathcal{B}$
\begin{equation}\label{ETA}
\eta_{n}:=\sup_{||h||\le 1}|Z(\mcl(h {\bf 1}_{B_{n}}))|,
\end{equation}
we must assume that
\begin{equation}\label{R1}
\lim_{n\rightarrow \infty}\eta_n=0,
\end{equation}
\begin{equation}\label{R2}
\eta_{n}||\mcl({\bf 1}_{B_{n}}\mu)||\le \text{const} \ \Delta_{n}.
\end{equation}
\end{itemize}

Notice that {\bf A0} and {\bf A1} are the content of the aforementioned Lemma 2.4 and Theorem 2.5 in \cite{DDHH}; it remains to prove {\bf A2} and {\bf A4}. The proofs, especially that of {\bf A2}, are quite long and we will defer them to the next sections.\\
\section{Proof of {\bf A2}}\label{A2}
We start by  noticing that the proof we present is also valid for the unperturbed operator, and this will be explicitly used in the following.
The proof is basically the same as the proof of Proposition 4.2 in \cite{DDHH}, with the difference that we allow subsets of the stable manifolds of length less than $\gamma_a.$ By density of $C^1(\mathcal{W}^u)$ in both $\mathcal{B}$ and $\mathcal{B}_w,$ it will be enough to take $h$ is such a smaller space. We have to control integrals of type: $\int_W \mcl_nh \varphi \,dm,$ where $W\in \Sigma$ and $\varphi \in C^1(W)$ (resp.\ $C^{\kappa}(W)$), according to the estimate of the weak (resp.\ strong) norm. \\
\subsection{Weak norm} 
Let us start with the weak norm and consider for instance $\mcl_n^3.$ We have
\begin{equation}\label{tor}
\int_W \mcl_n^3 h \varphi dm=\int_{W}\frac{(h{\bf1}_{\bn}{\bf1}_{\bn}\circ T {\bf 1}_{\bn}\circ T^2)(T^{-3}x) \varphi(x)}{\alpha^{-3}\gamma_a^3} dm(x).
\end{equation}
 We successively perform three changes of variable  along the backward images of $W$ each with Jacobian $\gamma_a,$ which will cancel the factor  $\gamma_a^3$ in the denominator in (\ref{tor}). But we have now to understand how those backward images are produced.
 
 Since $|W|\le \gamma_a,$ its inverse image will give  rise to at most two pieces $A_1, A_2$ of length respectively $a_1, a_2$ such that $a_1+a_2\le \frac{|W|}{\gamma_a}.$ But now $T^{-2}W$ is equal to $T^{-2}(A_1\cup A_2)$ and
   $T^{-1}(A_1)$ (resp. $T^{-1}(A_2)$) will produce the pieces $B_1, B_2$ (resp. $C_1, C_2).$ If we denote with $b_{1,2}, c_{1,2}$ the length of those pieces we have $b_1+b_2\le \frac{a_1}{\gamma_a}, c_1+c_2\le \frac{a_2}{\gamma_a}.$

  Our last step consists in iterating backward $B_{1,2}, C_{1,2};$ each of them will be expanded of a factor $\gamma_a,$  so we get
  $$
  (\ref{tor})=\sum_{*\in \{1,2\}}\alpha^3\int_{T^{-1}B_*}(h{\bf1}_{\bn}{\bf1}_{\bn}\circ T {\bf 1}_{\bn}\circ T^2)(x) \varphi(T^3x)dm(x)+
  $$
  $$
  \sum_{*\in \{1,2\}}\alpha^3\int_{T^{-1}C_*}(h{\bf1}_{\bn}{\bf1}_{\bn}\circ T {\bf 1}_{\bn}\circ T^2)(x) \varphi(T^3x)dm(x),
  $$
  where the measure $m$ is again the unnormalized Lebesgue measure.
   We now cut the eight  intervals $T^{-1}B_*, T^{-1}C_*$ into pieces of length $|W|.$ For instance $T^{-1}B_1$ will give $\frac{b_1}{\gamma_a|W|}$ pieces of length $|W|$ plus two pieces of length less than $|W|.$
But in the present case we have to add twice $3$ to those pieces for the presence of  \footnote{If we consider higher iterates of $\mcl,$ for instance of order $k,$ we should control terms like $W\cap \bn \cap T^{-1}\bn \cap \dots \cap T^{-(k-1)}\bn,$ where $W$ is a piece of stable manifold. Notice that each preimage $T^{-l}B_n, l=1,\dots,k-1,$ is contained in $2^l$ disjoint horizontal rectangles. Therefore $W$ could meet at most $k-1$ of such rectangles of different generation and hence at most $k-1$ preimages of $B_n.$ This implies that the complement in $W$ of such intersection is at most composed by $k$ connected intervals}$B_n^c$ .  Then we split the previous eight integrals over those smaller pieces which we denote with $\tilde{W}^{(n)}_j$, with $1\le j\le M_3,$ where $M_3=\frac{1}{\gamma_a|W|}(b_1+b_2+c_1+c_2)+2^3(1+3)$ is just an upper bound of   the number of the smaller  pieces after three backward iterations. Then we can write
$$
(\ref{tor})\le \alpha^3\sum_{j=1}^{M_3}\int_{\tilde{W}^{(n)}_j} h(x) \varphi(T^3x)dm(x)
$$

Using the preceding bounds on the couples $b_{1,2}, c_{1,2}, a_{L,R},$ we see that   $M_3 \le \frac{1}{|W|\gamma_a^3}|W|+2^3(1+3).$  By iterating backward $k$ times such a cardinality becomes $$M_k\le \frac{1}{\gamma_a^k}+2^{k}(1+k).$$

In order to compute the weak norm of $\mcl_n^3$ we must take a test function $\varphi$ verifying $|\varphi|_{C^{1}(W)}\le 1.$ If we now take two points $y_1, y_2\in \tilde{W}_j^{(n)}$ we have
$$\frac{|\varphi( T^3x)-\varphi( T^3y)|}{|x-y|}=\frac{|\varphi( T^3x)-\varphi( T^3y)|}{|T^3x-T^3y|} \frac{|T^3x-T^3y|}{|x-y|}\le H^{1}(\varphi)\gamma_a^{3}\le H^{1}(\varphi)
$$
 
where $ H^{1}(\varphi)$ is the H\"older exponent of $\varphi$ (on $W$). Therefore
$$
|\varphi\circ T^3|_{C^{1}(\tilde{W}_j^{(n)})}=|\varphi\circ T^3|_{C^0(\tilde{W}_j^{(n)})}+H^1(\varphi\circ T^3)\le 1.
$$
 By multiplying  and dividing the integral in (\ref{tor}) by $|\varphi\circ T^3|_{C^{1}(\tilde{W}_j^{(n)})}$ we finally get, for any $k\ge 1$ and remembering that $\alpha=1/2$:
\begin{equation}\label{fou}
|\mcl_n^k h|_w\le\left[ \left(\frac{\alpha}{\gamma_a}\right)^k+1+k \right] |h|_w.
\end{equation}
\begin{remark}\label{nnnn}
The kind of partitioning we consider above, namely by cutting the preimages into pieces of length $|W|$ was not really necessary to estimate the weak norm, but it is particularly adapted to control the strong stable norm, see below. For this reason we anticipated it here. We will see how one could have been  proceeded more directly in estimating the strong unstable norm even if the weaker bound one gets on the cardinality of the preimages will not improve the final result.
\end{remark}
\subsection{ Strong stable norm}\label{dfgj}
To compute  the strong stable norm, we closely follow the same calculations of section 4.1 in \cite{DDHH} and we write, still for the third iterate of the perturbed operator and using the notations above:
\begin{equation}\label{kri}
\int_W \mcl_n^3 h \varphi dm=\alpha^3\left\{\sum_{j}^{M_3}\int_{\tilde{W}_{j}^{(n)}}  h(y) [\varphi(T^3y)-\overline{\varphi_{j,n}}]dm(y)+\int_{\tilde{W}_{j}^{(n)}}  h(y)\overline{\varphi_{j,n}}dm(y)\right\},
\end{equation}
where
$$
\overline{\varphi_{j,n}}=\frac{1}{|\tilde{W}_{j}^{(n)}|}\int_{\tilde{W}_{j}^{(n)}}\varphi(T^3y)dm(y).
$$
Since $|\overline{\varphi_{j,n}}|\le \sup_{W}|\varphi|,$ we have immediately that the rightmost term in (\ref{kri}) is bounded by the right hand side of (\ref{fou}).
Instead the first piece on the right hand side is bounded by
\begin{equation}\label{rrr}
\alpha^3\sum_{j}^{M_3} \|h\|_{s}|\varphi\circ T^3-\overline{\varphi}_{j,n}|_{\tilde{W}_{j}^{(n)}, \kappa}=\sum_{j}^{M_3}\alpha^3 \|h\|_{s} |\tilde{W}_j^{(n)}|^{\kappa}|\varphi\circ T^3-\overline{\varphi}_{j,n}|_{C^{\kappa}(\tilde{W}_j^{(n)})}
\end{equation}
But $|\varphi \circ T^3-\overline{\varphi}_{j,n}|_{C^{\kappa}(W_j^{(n)})}= |\varphi \circ T^3-\overline{\varphi}_{j,n}|_{C^0(W_j^{(n)})}+\sup_{x\neq y}\frac{|\varphi( T^3x)-\varphi( T^3y)|}{|x-y|^{\kappa}}.$
We now treat the last term on the right hand side, the first giving the same result after having noticed that $|\varphi \circ T^3-\overline{\varphi}_{j,n}|= |\varphi( T^3x)-\varphi( T^3x^*)|$, being $x^*$ some point in $W_j^{(n)}$ by the mean value theorem, and  having multiplied and divided it by $\frac{|T^3x-T^3y|^{\kappa}}{|x-y|^{\kappa}}.$
We have: $$\frac{|\varphi( T^3x)-\varphi( T^3y)|}{|x-y|^{\kappa}}=\frac{|\varphi( T^3x)-\varphi( T^3y)|}{|T^3x-T^3y|^{\kappa}} \frac{|T^3x-T^3y|^{\kappa}}{|x-y|^{\kappa}}\le$$
$$
H^{\kappa}(\varphi)\gamma_a^{3\kappa}\le \gamma_a^{3\kappa} |\varphi|_{C^{\kappa}(W)}= \gamma_a^{3\kappa}|W|^{-\kappa}|\varphi|_{W,\kappa}\le \gamma_a^{3\kappa}|W|^{-\kappa},$$ where $ H^{\kappa}(\varphi)$ is the H\"older exponent of $\varphi$ (on $W$).
The sum (\ref{rrr}) is therefore bounded by
\begin{equation}\label{schi}
2\alpha^3\gamma_a^{3\kappa}|W|^{-\kappa}\sum_{j}^{M_3} \|h\|_{s} |W_j^{(n)}|^{\kappa}.
\end{equation}
By construction, all the intervals $|W_j^{(n)}|\le |W|\footnote{It is at this point where the partitioning we argued in remark \ref{nnnn} becomes useful.};$ by using the bound on the cardinality of such intervals given by $M_k,$ we finally get
$$
\|\mcl_n^k h\|_s\le \left[ \left(\frac{\alpha}{\gamma_a}\right)^k+1+k\right] |h|_w+ \|h\|_{s}  \left[2 (\alpha\gamma_a^{\kappa-1})^k+2\gamma_a^{k\kappa}(1+k)\right].
$$
\subsection{ Strong unstable norm}\label{ese}
In order to treat the strong  unstable norm, we follow section 4.3 in \cite{DL} adapted to our case, which is considerably much easier.
Therefore, take two stable manifolds $W_{1,2}$ at distance at most $\epsilon,$ and $\varphi_i$ on $W_i, i=1,2$ with $|\varphi_i|_{C^1(W_i)}\le 1.$
Call  $U_1\subset W_1$ and $U_2\subset W_2$ the connected intervals parametrized respectively by $(s_{W_1}, t), (s_{W_2}, t),$ with  $t$ belonging to the same interval. We call {\em matched} these two pieces. We call $V_{1,2}$  the two {\em unmatched} pieces in $W_{1,2};$ notice that the length of these two pieces is less than $\epsilon.$ Define now by  $U^{(j)}_{1, k}, U^{(j)}_{2,k}, j=1, \dots 2^k$ two preimages of order $k$ respectively of $U_1$ and $U_2$ with the same history, which means that if $s_{U^{(j)}_{1,k}}, s_{U^{(j)}_{2,k}}$ are the common ordinates of the points in respectively $U^{(j)}_{1,k}$ and $U^{(j)}_{2,k},$ then $s_{U^{(j)}_{1,k}}$ and $s_{U^{(j)}_{2,k}}$ belong to the same inverse branch of the map $T^k_Y$ given in (\ref{rty}).  Due to the linearity of the map, the sets $U^{(j)}_{1,k}$ and $U^{(j)}_{2,k}$ will be again matched and
$d(U^{(j)}_{1,k}, U^{(j)}_{2,k})=|s_{U^{(j)}_{1,k}}-s_{U^{(j)}_{2,k}}|\le \alpha^k d(U_1, U_2)\le \alpha^k \epsilon.$
Since $U^{(j)}_{1,k}$ and $U^{(j)}_{2,k}$ could contain each  at most $k$ preimages of the ball $B_n,$ we could have at most $2k$ matched intervals inside $U^{(j)}_{1,k}$ and $U^{(j)}_{2,k}$. Call $U^{(j, l)}_{1,k}$ and $U^{(j, l)}_{2,k}, l=1,\dots, 2k$  those smaller matched pieces.


The points of $U_{1,k}^{(j,l)}$ (resp. of $U_{2,k}^{(j,l)}$), will be parametrized as $(s_{1,k}^{(j,l)},t), t\in U_{1,k}^{(j,l)}$ (resp. $(s_{2,t}^{(j,l)}, t\in U_{2,k}^{(j,l)}$))\footnote{With abuse of notation, $U_{1,k}^{(j,l)}$ and $U_{2,k}^{(j,l)}$ denote the segments of stable manifolds, where $s_{1,t}^{(j,l)}$ (resp. $s_{2,t}^{(j,l)}$) is the common ordinate of the points in $U_{1,k}^{(j,l)}$ (resp. in $U_{2,k}^{(j,l)}$)}. We have to control pieces of the type
\begin{equation}\label{pio}
\frac{1}{\eps^{\beta}}\left( \int_{U_{1,k}^{(j,l)}} h(s_{1,k},t)\varphi_1(T^k(s_{1,k},t))dt- \int_{U_{2,k}^{(j,l)}} h(s_{2,k},t)\varphi_2(T^k(s_{2,k},t))dt
 \right).
\end{equation}
we now introduce the auxiliary function
$$
\overline{\varphi}_2(s_{W_2},t)=\varphi_1(s_{W_1},t), t\in U_1\footnote{Same convention for $U_1$ as in the previous footnote.}.
$$
Then we bound  (\ref{pio}) as
$$
\frac{1}{\eps^{\beta}}\left|\int_{U_{1,k}^{(j,l)}} h(s_{1,k}^{(j,l)},t)\varphi_1(T^k(s_{1,k}^{(j,l)},t))dt- \int_{U_{2,k}^{(j,l)}} h(s_{2,k}^{(j,l)},t)\overline{\varphi}_2(T^k(s_{2,k}^{(j,l)},t))dt
 \right|+
$$
$$
\frac{1}{\eps^{\beta}}\left|\int_{U_{1,k}^{(j,l)}} h(s_{2,k}^{(j,l)},t)[\overline{\varphi}_2(T^k(s_{2,k}^{(j,l)},t))dt- \varphi_2(T^k(s_{2,k}^{(j,l)},t))]dt
 \right|= (I)+(II)
$$
We  begin to treat the first piece (I); notice that $T^k(s_{2,k}^{(j,l)},t)$ is a point of the form $(s_{W_2}, t^*), t^*\in U_2$ and therefore $\overline{\varphi}_2(T^k(s_{2,k}^{(j,l)},t))=\overline{\varphi}_2(s_{W_2}, t^*)=\varphi_1(s_{W_1}, t^*)=\varphi_1(T^k(s_{1,k}^{(j,l)},t)),$ since the points $(s_{1,k}^{(j,l)}, t)$ and $(s_{2,k}^{(j,l)}, t)$ are aligned on the same vertical line. Notice now that $|\varphi_1\circ T^k|_{C^1(U_{1,k})}\le 1,$ see a similar computation done for the strong stable norm, and moreover $d_q(\varphi_1\circ T^k, \varphi_2\circ T^k)=0.$ We also have $d(U_{1,k}^{(j,l)}, U_{2,k}^{(j,l)})\le \alpha^{k}\eps<\eps,$ which finally implies
$(I)\le \alpha^{k\beta}\|h\|_u.$ We now pass to estimate (II) using the strong stable norm as
$$
(II)\le \frac{1}{\eps^{\beta}}\|h\|_s|U_{1,k}^{(j,l)}|^{\kappa}|\varphi_2\circ T^k-\varphi_1\circ T^k|_{C^{\kappa}(U_{1,k})}.
$$
We now have, using estimates as above
$$
|\varphi_2\circ T^k-\varphi_1\circ T^k|_{C^{\kappa}(U_{1,k})}=|\varphi_2\circ T^k-\varphi_1\circ T^k|_{C^0(U_{1,k})}+
$$
$$
\sup_{y_1,y_2\in U_{1,k}^{(j,l)},x\neq y}\frac{|\varphi_2\circ T^k(y_1)-\varphi_1\circ T^k(y_1)-\varphi_2\circ T^k(y_2)+\varphi_1\circ T^k(y_2)|}{|y_1-y_2|^{\kappa}}\le
$$
$$
|\varphi_2-\varphi_1|_{C^0(U_1)}+\gamma_a^{k\kappa} H^{\kappa}(\varphi_1-\varphi_2)=d_{\kappa}(\varphi_1, \varphi_2)\le \eps.,
$$
where $H^{\kappa}$ is computed on $U_1.$ Therefore
$$
(II)\le \eps^{1-\beta}\gamma_a^{-k\kappa}\|h\|_s.
$$



For the unmatched pieces, we have to take into account those generated by the $2^k$ preimages of $V_{1,2},$ but also the unmatched pieces in the $U^{(j)}_{m,k}, m=1,2, j=1,\dots, 2^k.$

Let us consider first the pieces generated by the $2^k$ preimages of $V_{1,2};$ their total number is at most $2k2^k.$ If we call $V_{(k)}$ one of them and supposing it belongs to the backward images of
$W_1$,  we must estimate  the strong stable norm of the quantity
 $\frac{1}{\epsilon^{\beta}}\left|\int_{V_{(k)}}h(y)\varphi(T^ky)dm(y)\right|.$ We multiply it by  $|V_{(k)}|^{\kappa}|\varphi\circ T^k|_{C^{\kappa}(V_k)}$. But $|\varphi\circ T^k|_{C^{\kappa}(V_{(k)})}\le |\varphi|_{C^0(W_1)}+H^{\kappa}(\varphi)\gamma_a^{k\kappa}\le 1,$ and $|V_{(k)}|^{\kappa}\le \epsilon^{\kappa}\gamma_a^{-k\kappa}.$ We now consider the unmatched pieces into the $U_{m,k}^{(j)}, m=1,2.$ These are generated by the intersections of the preimages of the ball $B_n$ with the preimages of $W_{1,2}.$ These intersections could affect one or both of the stable segments $U^{(j)}_{m, k}.$ As soon as an intersection occurs, it could generate at most three unmatched pieces (the intersection itself and two short segments on both sides of the intersection); therefore we could have at most $6k$ unmatched pieces for the couple  $U^{(j)}_{m, k}.$ About their size, we use the convexity argument given in section 6.3 in \cite{DL}. If an intersection occurs with one or both the  $U^{(j)}_{m, k},$ it also happens between the ball $B_n$ and some backward iterate of order $l\le k$ of the couple $W_{1,2}.$ In this case the intersection will be of order $\sqrt{\epsilon}$, namely $C_{B_{n}}\sqrt{\epsilon}$ (our $B_n$ is a real ball), where the constant $C_{B_{n}}<1$ depends on the radius of $B_n,$  and therefore the backward intersections with $U^{(j)}_{m, k}$ will be of order $(\sqrt{\epsilon} \gamma_a^{-1})^{k}.$
Putting together these contribution and asking for
$$
\kappa>2\beta
$$
we have, since $\alpha^k=2^{-k}:$
 $$
 \| \mcl_n^k h \|_u\le2k\alpha^{k\beta}\|h\|_u+12k \gamma_a^{-k\kappa}\|h\|_s
 $$

In conclusion we get for $k\ge 1:$
\begin{equation}
\| \mcl_n^k h \|=\| \mcl_n^k h \|_s+b\| \mcl_n^k h \|_u\le  
\end{equation}
\begin{equation}\label{tre}
\left[ \left(\frac{\alpha}{\gamma_a}\right)^k+1+k\right]|h|_w+\left[2 (\alpha\gamma_a^{\kappa-1})^k+2\gamma_a^{k\kappa}(1+k)\right]\|h\|_s+b\left(2k\alpha^{k\beta}\|h\|_u+12k \gamma_a^{-k\kappa}\|h\|_s\right).
\end{equation}
We now put $g_k:=\left[ \left(\frac{\alpha}{\gamma_a}\right)^k+1+k\right]$ and $u_{\kappa}:=\alpha\gamma_a^{\kappa-1}.$ Then we ask that $u_{\kappa}<1$ which needs
$$
\kappa>1-\frac{\log \alpha}{\log \gamma_a}.
$$
Then can rewrite  (\ref{tre}) as
\begin{equation}\label{quatro}
\| \mcl_n^k h \|\le g_k|h|_w+[2(2+k)u^k_{\kappa}+12bk\gamma_a^{-k\kappa}]\|h\|_s+2bk\alpha^{k\beta}\|h\|_u.
\end{equation}
Then choose $b$ such that\footnote{We will see in a moment that this choice will be done for a particular $k.$}
$$
b<\frac{(2+k)u^k_{\kappa}}{k\gamma_a^{-k\kappa}},
$$
which allow us to rewrite (\ref{quatro}) as
$$
\| \mcl_n^k h \|\le g_k|h|_w+4(2+k)u^k_{\kappa}\|h\|_s+2bk\alpha^{k\beta}\|h\|_u.
$$
Then pose
$$
\sigma=\max\left(u_{\kappa}, \alpha^{\beta}\right)<1,
$$
which gives
$$
\| \mcl_n^k h \|\le g_k|h|_w+r_k\sigma^k\|h\|,
$$
where we set $r_k=4(2+k).$\\
We now fix a value of $k,$ say $k_0,$ such that
$$
\rho:=(r_{k_0}\sigma^{k_0})^{\frac{1}{k_0}}<1
$$
and we replace $k$ with $k_0$ in the bound above for $b.$ With these positions and by using
blocks of length $k_0,$ it is immediate to rewrite (\ref{tre}) as, for any $k>0:$
$$
\| \mcl_n^k h \|\le \rho^k|h|_w+M^k\|h\|,
$$
where
$M:=g_{k_0}^{\frac{1}k_0}(1-r_{k_0}\sigma^{k_0})^{-1},$ and this proves (\ref{LLYY}).\\
\begin{remark}\label{constant}
   We summarize the bounds we imposed on the relevant quantities we used up to now; we have, since $\alpha=1/2:$
   \begin{itemize}
    \item $0<\beta<1-\kappa$  and $2\beta<\kappa.$  This requires first  $\beta<1/3.$
    \item $\beta+q<1,$ with $q\in (0,1),$ and $q>5/6,$ which  implies $\beta<1/6.$
    \item $\kappa>1-\frac{\log \alpha}{\log \gamma_a}=1+\frac{\log 2}{\log \gamma_a}.$
    \item Finally we will see below that $\kappa>\frac{\alpha\log \alpha^{-1}+(1-\alpha)\log (1-\alpha)^{-1}}{\log \gamma_a^{-1}}=-\frac{\log 2}{\log \gamma_a.}$ \\ This is  verified by several couples of the parameters $\kappa, \gamma_a.$ For instance, for any $1/2<\kappa<1,$ we could take $\gamma_a=1/4.$ Alternatively, by choosing $\kappa=3/4,$ we could take $2^{-4}<\gamma_a<2^{-4/3}.$
   \end{itemize}
\end{remark}
{\bf Warning: from now on we will consider the baker map (\ref{bibi}) with the parameters $\alpha=v=0.5$ and $\gamma_a$ satisfying the constraints given in the previous remark \ref{constant}.}



\section{Proof of {\bf A4}}
We now pass to  justify {\bf A4.} We remind that $Z$  is the unique solution of the eigenvalue equation
 $\mcl^*Z=Z,$ where $\mcl^*$ is the dual of the transfer operator. By setting
\begin{equation}\label{DUAL}
Z(h):=h(1), \ h\in \mathcal B,
\end{equation}
we  have for $h\in \mathcal B$:
$$
\mcl^*Z(h)=Z(\mcl h)= (\mcl h)(1)= h(1\circ T)=h(1)=Z(h).
$$
Coming back to $\Delta_{n}$ we see immediately that
\begin{equation}\label{DELTA2}
\Delta_{n}=Z(\mcl({\bf 1}_{B_{n}}\mu))=\mcl({\bf 1}_{B_{n}}\mu)(1)=\int {\bf 1}_{B_{n}}\,d\mu=\mu(B_{n}).
\end{equation}
The term $||\mcl({\bf 1}_{B_{n}}\mu)||$ can be handled very easily using the Lasota-Yorke inequality which we proved in item {\bf A2}  above. It follows in fact from (\ref{LLYY}) that there are two constants $C_1, C_2$  depending only on the map such that
$$
||\mcl({\bf 1}_{B_{n}}\mu)||\le C_1 ||{\bf 1}_{B_{n}}\mu||+C_2|{\bf 1}_{B_{n}}\mu|_w.
$$
\begin{lemma}\label{bcv}
There exists two constants $\hat{C}_1, \hat{C}_2$ independent of $n$ such that 
\begin{equation} \label{vb}
 ||{\bf 1}_{B_{n}}\mu||\le \hat{C_1}||\mu||\;\;\; \text{and} \;\;\; |{\bf 1}_{B_{n}}\mu|_w\le  \hat{C_2}|\mu|_w.
 \end{equation}
 \end{lemma}
 \begin{proof} 
 The proof follows closely the  arguments given in section \ref{subsub2} to bound the quantity ${\bf 1}_{B^{(k)}}h,$ but we should now be more careful in getting constants which do not depend about upon $B_n.$ Notice that, contrarily to section \ref{subsub2}, we take here the intersection of $W$  with the ball $B_n,$ not with its preimages. With this in mind, it is immediate  to check that  $||{\bf 1}_{B_{n}}\mu||_s\le ||\mu||_s\;\; \text{and} \;\; |{\bf 1}_{B_{n}}\mu|_w\le  |\mu|_w.$ It remains to compute the strong unstable norm and this reduces to bound the difference, for a smooth $h$: $\textfrak{D}_n:=\frac{1}{\epsilon^{\beta}}\textfrak{I}_n,$ where $\textfrak{I}_n:=\left|
 \int_{W_1\cap B_n}h\varphi_1 dm- \int_{W_2\cap B_n}h\varphi_2 dm\right|,$ and $W_1 W_2, \varphi_1, \varphi_2$ verify the constraints given in (\ref{sunorm}). We split the argument in two parts. We call $r_n$ the radius of the ball $B_n$ and  we begin to take $\epsilon\ge 0.5 r_n.$ Then we have the rough bound 
 $$
 \textfrak{I}_n\le 2\ r_n \ \|h\|_s \ \Rightarrow \  \textfrak{D}_n\le 2^{1+\kappa}\|h\|_s \epsilon^{\kappa-\beta}<2^{1+\kappa}\|h\|_s,
 $$
 since $\beta<\kappa.$
 Then we consider $\epsilon <0.5 r_n;$ we split the difference in  $\textfrak{I}_n$  over unmatched and matched pieces. There could be at most one matched piece inside $B_n$ giving the contribution $\|h\|_u.$ If the matched piece is inside $B_n$ there could be at most two unmatched pieces. They have length $\le\epsilon$ if they are on the extremities of the two stable manifolds inside $B_n$.  Otherwise they are generated when the two stable manifolds meet  the boundary of $B_n.$ It is a simple exercise to see that the sum of the  lengths of those unmatched pieces is bounded by the maximum difference of the lengths of two horizontal chords whose vertical distance is $\epsilon$ and that value is less or equal to $2\sqrt{2r_n\epsilon}\le 2\sqrt{2\epsilon}.$ 
 
  Finally there could be an unmatched piece when  only one manifold crosses $B_n$ and for the same argument as above its length is bounded by $2\sqrt{2\epsilon}.$ Summing all those contribution when $\epsilon<0.5 r_n$ we get
  $$
   \textfrak{I}_n\le \epsilon^{\beta}\|h\|_u+2\epsilon^{\kappa} \|h\|_s+2^{1+\frac{\kappa}{2}}\epsilon^{\frac{\kappa}{2}}\|h\|_s \Rightarrow \  \textfrak{D}_n\le \|h\|_u+2\|h\|_s+2^{1+\frac{\kappa}{2}}\|h\|_s,
  $$
  since $2\beta<\kappa.$ 
 \end{proof}
 
By setting
$$
C_3:= C_1 \hat{C}_1||\mu||+C_2\hat{C}_2|\mu|_w,
$$
we are led to prove that (see (\ref{R2})), $\eta_{n}C_3\le \text{const}\ \Delta_{n}$, namely
\begin{equation}\label{R3}
\eta_{n}\le \ \text{const}\ \Delta_{n}= \text{const} \ \mu(B_{n}).
\end{equation}
Before continuing, we have to focus on $\mu(B_n)=\mu(B(z, e^{-u_n})).$ It is well known that for $\mu$-almost all $z$ and by taking the radius sufficiently small, depending on the value $\iota,$
$e^{-u_n (d+\iota)}\le \mu(B(z, e^{-u_n})\le e^{-u_n (d-\iota)},$ where $\iota>0$ is arbitrarily small. This follows from the existence of the limit
\begin{equation}\label{ED}
\lim_{r\rightarrow 0^+} \frac{\log \mu(B(x,r))}{\log r}=d, \ \text{for} \ x \ \text{chosen $\mu$-a.e.},
\end{equation}
and quantity $d$ is the Hausdorff dimension of the measure $\mu$ which  in our case  reads \cite{EO}, eq. (3.24):
$$
d=1+d_s, \ \text{where} \ d_s:=\frac{\alpha\log \alpha^{-1}+(1-\alpha)\log (1-\alpha)^{-1}}{\log \gamma_a^{-1}}.
$$

We now have:
\begin{lemma}\label{ldc}
Assume
$\kappa>d_s$.

Then
$$
\eta_n\le 2 \lceil{\gamma_a^{-1}}\rceil \mu(B_n).
$$
\end{lemma}
\begin{proof}
By density it will be enough to prove the Lemma for $h\in C^1(X).$ We have
$$
Z(\mcl(h\ {\bf 1}_{B_{n}}))=\int h\ {\bf 1}_{B_{n}} dm.
$$
 By disintegrating along the stable partition $\mathcal{W}^s$ we get:
\begin{equation}\label{vb}
\int h\ {\bf 1}_{B_{n}} \,dm_L
=\int_{\xi}d\lambda(\xi)\left[\int_{W_{\xi}}({\bf 1}_{B_{n}} h)(x)\,dm(x)\right].
\end{equation}
We now cut the stable manifold $W_{\xi}$ in pieces of length $\gamma_a$ in order to compute the strong stable norm on each of them and we put $|\tilde{W}_{\xi}|$ the largest intersection of such pieces with $B_n.$  We immediately get
\begin{eqnarray*}
(\ref{vb})\le& \int_{\xi}d\lambda(\xi)\left[\lceil{\gamma_a^{-1}}\rceil|\tilde{W}_{\xi}|^{\kappa} \|h\|_s\right]
\le  e^{-u_n\kappa}\lceil{\gamma_a^{-1}}\rceil ||h||_s\lambda(\xi; B_{n}\cap W_{\xi}\neq \emptyset),
\end{eqnarray*}
where $\lambda$ is the quotient measure  on the space of stable leaves $ W_{\xi}$ belonging to $\mathcal{W}^s$; and indexed by $\xi$, see for instance \cite{viana}, Appendix A. By definition of disintegration we have that
$$
\lambda(\xi; B_{n}\cap W_{\xi}\neq \emptyset)=m_L(\bigcup W_{\xi},  B_{n}\cap W_{\xi}\neq \emptyset)=2e^{-u_n},
$$
 and therefore
$$
\eta_n\le 2 \lceil{\gamma_a^{-1}}\rceil e^{-u_n(\kappa+1)}.
$$

 We finally  have
$$
\eta_n\le 2 \lceil{\gamma_a^{-1}}\rceil e^{-u_n(\kappa+1)}\le 2\lceil{\gamma_a^{-1}}\rceil e^{-u_n (d+\iota)}\le 2\lceil{\gamma_a^{-1}}\rceil\mu(B_n),
$$
provided we choose
\begin{equation}\label{star}
\kappa>d+\iota-1
\end{equation}
which can be satisfied by assumption.
\end{proof}
\begin{remark}\label{RRR1}
The local comparison between the Lebesgue and the SRB measure of a ball of center $z$ obliged us to choose $z$ $\mu$-almost everywhere because in this way we have a precise value for  the locally constant dimension $d$. We are therefore discarding several points, possibly periodic, where the limiting distribution for the Gumbel law (see next section) could exhibit extremal indices different from $1.$
\end{remark}
\begin{remark}\label{RRR2}
 For invertible, piecewise differentiable hyperbolic maps   in dimension $2$, the construction of the Banach space imposes that $\kappa<1;$   for billiard maps associated with Lorentz gases, \cite{DZ1}, it even verifies $\kappa\le 1/6.$ This could make difficult to check condition (\ref{star}) for invariant sets with large $d,$ like   Anosov diffeomorphisms for instance. In some sense this difficulty was already raised in section 4.5 in the Keller's paper \cite{GK}, where an estimate like ours in terms of the H\"older exponent $\kappa$ was given and the subsequent question of the comparison with the SRB measure was addressed.
\end{remark}
\section{The limiting law}
\subsection{Gumbel law}\label{fc}
We have now all the tools to compute the asymptotic behavior  of $\mcl_n.$ We need one more ingredient which will constitute our last assumption:
\begin{itemize}
\item {\bf A5} Let us suppose that the following limit exist for any $k\ge 0:$
\begin{equation}\label{QQ}
q_{k}=\lim_{n\rightarrow \infty}q_{k,n}:=\lim_{n\rightarrow \infty} \frac{Z\left([(\mcl-\mcl_n)\mcl_n^k(\mcl-\mcl_n)]\mu\right)}{\Delta_n}
\end{equation}
\end{itemize}
Notice that
$$
q_{k,n}=\frac{\mu(B_n\cap T^{-1}B_n^c\cap \cdots \cap T^{-k}B_n^c\cap T^{-(k+1)}B_n)}{\mu(B_n)}
$$
and therefore by the Poincar\'e recurrence theorem
$$
\sum_{k=0}^{\infty}q_{k,n}=1.
$$
Therefore if the limits (\ref{QQ}) exist, the quantity
\begin{equation}\label{EEII}
\theta=1-\sum_{k=0}^{\infty}q_k,
\end{equation}
is well defined and verifies
$$
0\le \theta\le 1.
$$
It is   called the {\em extremal index} and it   modulates the exponent of the Gumbel law as we will see in a moment. We have in fact by Theorem 2.1 of \cite{KL}:
$$
\lambda_n=1-\theta\Delta_n=\exp(-\theta\Delta_n+o(\Delta_n)),
$$
or equivalently
$$
\lambda_n^n=\exp(-\theta n\Delta_n+no(\Delta_n)).
$$
Therefore we have
$$
\mu(M_n\le u_n)=\mcl_n^n\mu(1)=\lambda_n^n[\mu_n(1)Z_n(\mu)+Q^n_n(\mu)(1)]
$$
and consequently
$$
\mu(M_n\le u_n)=\exp(-\theta n\Delta_n+no(\Delta_n))[O(1)+ Q^n_n(\mu)(1)],
$$
since  $\mu_n(1)=1$ and
it has been proved in \cite{KL}, Lemma 6.1, $Z_n(\mu)\rightarrow 1$ for $n\rightarrow \infty.$ At this point we need an important assumption, which basically reduces to fixing the sequence $u_n$ and allow us to get a non-degenerate limit for the distribution of $M_n$. We in fact ask that
\begin{equation}\label{LLSS}
n\ \Delta_n\rightarrow \tau, \ n\rightarrow \infty,
\end{equation}
where $\tau$ is a positive real number. With this assumption, using (\ref{pif}) and (\ref{ggg}), we have
$$
|Q_n^n(\mu)(1)|\le \text{const}\  sp(Q)^n ||\mu||\rightarrow 0.
$$
In conclusion we get the Gumbel  law
$$
\lim_{n\rightarrow \infty}\mu(M_n\le u_n)=e^{-\theta \tau}.
$$

\subsection{The extremal index}\label{43}
We are now ready to compute the $q_{k,n},$ which will determine the extremal index. Let us first suppose that the center of the ball $B_n$ is not a periodic point; then the points $T^j(z), j=1,\cdots, k$ will be disjoint from $z.$ Let us take the ball so small that is does not cross the set $T^{_j}\Gamma, j=1,\cdots,k,$ where $\Gamma$ is the discontinuity line $(y=\alpha).$ In this way the images of $B_n$  will be ellipses with the long axis along the unstable manifold and the short axis stretched by a factor $\gamma^k.$ By continuity and taking $n$ large enough, we can manage that all the iterates of $B_n$ up to $T^k$ will be disjoint from $B_n$ and for such $n$ the numerator of $q_{k,n}$ will be zero.
At this point we can state the following result:

\begin{proposition}\label{ty1}
Let $T$ be the baker's transformation and consider the function $M_n (x):=\max\{\Xi(x), \dots, \Xi(T^{n-1}x)\},$ where $\Xi(x)=-\log d(x,z),$ and $z$ is chosen $\mu$-almost everywhere with respect to the SRB measure $\mu.$ Then, if $z$ is not periodic, we have
$$
\lim_{n\rightarrow \infty}\mu(M_n\le u_n)=e^{-\tau},
$$
where the boundary level $u_n$ is chosen to satisfy $n\mu(B(z, e^{-u_n}))\rightarrow \tau,$ for some positive $\tau.$
\end{proposition}

Suppose now $z$ is a periodic point of minimal period $p;$ of course the next considerations make sense if the limit (\ref{ED}) exists. By doing  as above we could stay away from the discontinuity lines up to $p$ iterates and look simply to $T^{-p}(B_n)\cap B_n.$ Since the map acts linearly, the  $p$ preimage of $B_n$ would be an ellipse with center $z$  and symmetric w.r.t.\ the unstable manifold passing trough $z.$ So we have to compute the SRB measure of the intersection of the ellipse with the ball shown in Fig.~2.

It turns out that this computation is not easy. The natural idea would be  to disintegrate
 the SRB measure along the unstable manifolds  belonging to the unstable partition
  $\mathcal{W}^u$. We index such fibers as $W_{\nu}$ and we put $\zeta(\nu)$
   the associated quotient measure. Let us recall that the conditional measures
    along leaves $W_{\nu}$ are normalized Lebesgue measures: we denote them with $l_{\nu}.$ If we call $\mathcal{E}_{in}$ the region of the ellipse inside the ball $B_n$, we have to compute
    \begin{equation}\label{cdc}
    \frac{\int l_{\nu}(\mathcal{E}_{in}\cap W_{\nu})\,d\zeta(\nu)}{\int l_{\nu}(B_n\cap W_{\nu})\,d\zeta(\nu)}.
    \end{equation}
    Although simple geometry allows us to compute easily the length of $\mathcal{E}_{in}\cap W_{\nu}$ and $B_n\cap W_{\nu},$ and since they vary with $W_{\nu},$ it is not at the end clear how to perform the integral with respect to the quotient  measure, especially because we need  asymptotic estimates, not bounds. We therefore proceed  by introducing a different metric, a nice trick which was already used in \cite{NF}. We use the $l^{\infty}$ norm on $\mathbb{R}^2$ for which  $|(x,y)|_{\infty}=\max\{|x|, |y|\}.$ In this way the ball $B_n$ will become a square with sides of length $r_n:=e^{-u_n}$ and $T^{-p}(B_n)$ will be a rectangle with  the long side of length $\gamma_a^{-p}r_n$ and the short side of length $\alpha^p r_n.$ This rectangle will be placed symmetrically with  respect to the square as indicated in Fig. 3.   The ratio (\ref{cdc}) can now be computed easily since the length in the integrals are constant and we get $\alpha^p.$ We will see that the value computed in this way is the right one, see Proposition \ref{ty2}, but in principle we {\em cannot apply the spectral technique} since the geometric shape of the rectangles does not allow to show that the characteristic function of such rectangles is in $\mathcal{B}$ and also it does not fit the convexity requirement which we used to control the unmatched pieces in the strong unstable norm. We will introduce in section \ref{eds} below a different technique which will allow us to get the extremal index even when the target sets are rectangles.


\begin{figure}\begin{tikzpicture}[scale=3]
\draw[fill=blue!30] (0,0) circle (0.5cm);
\draw[fill=green] (0,0) ellipse (0.80cm and 0.25cm);
\draw[dashed] (0,0) circle (0.5cm);
\draw[->] (0,-1.0) -- (0,1.0);
\end{tikzpicture}
\caption{Computation of the extremal index around periodic point with the Euclidean metric. The vertical line is an unstable manifold. We should compute the green area inside the circle. }
\label{tik}
\end{figure}
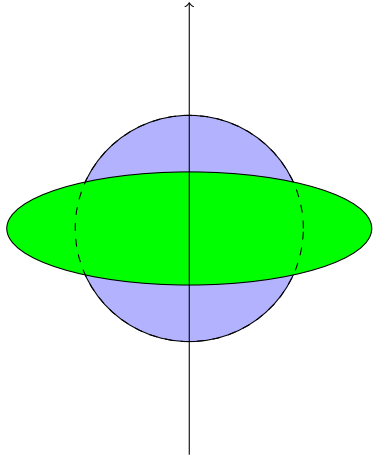

\begin{figure}\begin{tikzpicture}[scale=3]
\draw[fill=blue!30] (0,0) rectangle (1.2cm,1.2cm);
\draw[fill=green] (-0.3cm,0.3cm) rectangle (1.5cm,0.9cm);
\draw[dashed] (0,0) rectangle (1.2cm,1.2cm);
\end{tikzpicture}
\caption{Computation of the extremal index around periodic point with the $l^{\infty}$ metric.  We should compute the green area inside the square. }
\label{tik}
\end{figure}
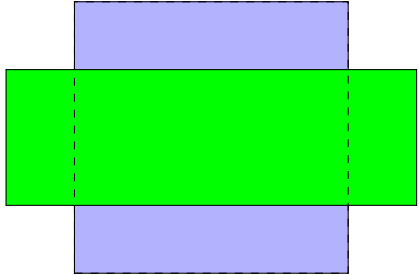

\section{Poisson statistics}\label{PPP}
\subsection{The spectral approach} As mentioned in the introduction, the spectral technique has been recently generalised  to study the statistics of the number of visits in balls shrinking around a point \cite{compound}. We briefly introduce such an approach and the reader will see that we can easily adapt it to the baker's map. The starting point is to consider   the following counting function
\begin{equation*}
N^{\tau}_{B_n}(x)=\sum_{i=0}^{\floor{\tau/\mu(B_n)}} {{\bf 1}_{B_n}} \circ T^i(x),
\end{equation*}
 where $\tau$ is a positive
parameter and
$x\in X.$ The goal is to study the distribution of this discrete random variable in the limit $n\rightarrow \infty;$ with the spectral approach we will rather look at the characteristic function of such a variable.\\ We begin to define $S_{n,k}:=\sum_{i=0}^{k} {{\bf 1}_{B_n}} \circ T^i$ and put $S_{n,(\tau,n)}:=N^{\tau}_{B_n}.$ We then define the perturbed operator
$$
\mcl_{n,s}(h)=\mcl(e^{is{\bf 1}_{B_n}}h), \ s\in \mathbb{R}, \ h\in \mathcal{B}.
$$
A simple computation shows that
$$
\mcl_{n,s}^k(\mu)(1)=\int e^{isS_{n,k}}d\mu,
$$
which suggests to get information on the characteristic function of $S_{n,k}$ by the behavior of the top eigenvalue $\lambda_{n,s}$ of the perturbed operator $\mcl_{n,s}.$ At this point the analysis proceeds in the same manner as for  the perturbed operator $\mcl_n$ and we sketch here the main steps. The {\em difference} between the two operators is now quantified by
$$
\Delta_{n,s}:=Z(\mcl-\mcl_{n,s})(\mu)=(1-e^{is})\mu(B_n),
$$
and
\begin{equation}\label{eee}
\lambda_{n,s}=1-\tilde{\theta}(s)(1-e^{is})\mu(B_n)+o(\mu(B_n)).
\end{equation}
The quantity $\tilde{\theta}(s)$ plays the  role of the extremal index  and is defined according to the formula (\ref{QQ}), which in the present case reduces to $\tilde{\theta}(s)=1-\sum_{k=0}^{\infty}q_k(s),$ where
\begin{eqnarray}
q_k(s) =\lim_{n\rightarrow \infty}\frac{1}{1-e^{is}}\sum_{\ell=0}^k(1-e^{is})^2e^{i\ell s}\beta_n^{(k)}(\ell)
=(1-e^{is})\sum_{\ell=0}^ke^{i\ell s}\beta_k(\ell),\\\label{sh}
\beta_n^{(k)}(\ell):=\frac{\mu(x; x\in B_n, T^{k+1}(x)\in B_n, \sum_{j=1}^k {\bf 1}_{B_n}(T^jx)=\ell)}{\mu(B_n)}.
\end{eqnarray}
and we  suppose that the limit
$
\beta_k(\ell):=\lim_{n\rightarrow \infty}\beta_n^{(k)}(\ell)
$
exists. Then we have
$$
\tilde{\theta}(s)=1-(1-e^{is})\sum_{k=0}^{\infty}\sum_{\ell=0}^ke^{i\ell s}\beta_k(\ell),
$$
and the exponential decay of correlation of the measure $\mu$ allows us to show that the series $\sum_{k=0}^{\infty}\sum_{\ell=0}^k\beta_k(\ell)$ converges absolutely\footnote{See section 3 in \cite{compound} for the proof of this convergence which applies to our case as well.} and therefore $\tilde{\theta}(s)$ is $C^{\infty}$ in the neighborhood of $0.$ If now return to the eigenvalue  (\ref{eee}), we exponentiate it to the power $n$ and we use again the threshold condition (\ref{LLSS}), $n\mu(B_n)\rightarrow \tau,$ we finally get
$$
\lim_{n\rightarrow \infty}\int e^{is S_{n,(\tau,n)} }d\mu=e^{-\tilde{\theta}(s)(1-e^{is})\tau}:=\Sigma(s).
$$
Since $\Sigma(s)$ is continuous in $s=0,$ it is the characteristic function of some random variable $\mathfrak{Z},$ possibly defined on a different probability space $(\Omega, \mathcal{F}, \mathbb{P})$. The variable $\mathfrak{Z}$ is clearly non-negative and integer valued and it is also infinitely divisible since $e^{-\tilde{\theta} (s) (1-e^{is}) \tau}=(e^{-\tilde{\theta}(s) (1-e^{is}) \tau/m})^m,$ for any $m.$ This implies that $\mathfrak{Z}$ has a compound Poisson (CP) distribution, see \cite{feller} or \cite{compound} for more references, namely it  may be written as $
\mathfrak{Z}:=\sum_{j=1}^N X_j$,
where the $X_j$ are  i.i.d. random variables  defined on  same probability space, and $N$ is Poisson distributed with {\em intensity} $\varkappa$ and $X_j$ has distribution $\mathbb{P}(X_j=l)=\rho_l;$ moreover $N$ is independent of all $X_j.$ We call the sequence $(\rho)_{l\ge 1}$ the {\em cluster size distribution} of $Z.$ Among the CP distributions, two are particularly important, the standard Poisson distribution and the P\`olya-Aeppli distribution. For the standard Poisson $\rho_1=1;$ for P\`olya-Aeppli the distribution of $X_j$ is geometrical, namely $\rho_l=\eta(1-\eta)^l, \eta\in (0,1).$ For such  distributions the associated characteristic functions are perfectly known.  To determine them for our baker' system one should prove the existence and compute the quantities (\ref{sh}), which are of geometric and dynamical nature.  This will be done in the next section in the context of a more probabilistic approach to Poisson-like statistics. Actually the quantities computed in the next section are not exactly those in (\ref{sh}),  but it is not difficult to modify their derivation to get (\ref{sh}) and therefore reprove Proposition \ref{limiting.distribution} with the spectral approach. As we said in the introduction and in section \ref{43}, we will present the alternative  probabilistic approach since it will allow us to cover a wider class of target sets and also to get the example \ref{IIEE} which shows a CP distribution different from the standard Poisson and the  P\`olya-Aeppli.

\subsection{The probabilistic approach}\label{eds} We now use a recent technique  developed in \cite{HHVV} and apply it to  our baker's map. We will recover the usual dichotomy and get a pure Poisson distribution when the points are not periodic, and a P\'olya-Aeppli distribution around periodic points with the parameter giving the geometric distribution of the size of the clusters which coincide with the extremal index computed in the preceding section. This last result is achieved in particular if we use the $l^{\infty}$ metric.  This result is not surprising; what is interesting is the great flexibility of the technique of the proof which allows us to get easily the expected properties. In order to apply the theory in \cite{HHVV}, we need to verify a certain number of assumptions, but otherwise defer to the aforementioned  paper for precise definitions. Here we recall the most important requirements and prove in detail one of them.  \\
{\em Warning:} the next considerations are carried  over with the Euclidean metric which is more natural for applications. As for the  visits to periodic points we will use the $l^{\infty}$ metric and the following computations are even easier.\\

{\em Decay of correlation.} There exists a decay function $\mathcal{C}(k)$ so that
$$
\left|\int_MG(H\circ T^k)\,d\mu-\mu(G)\mu(H)\right|
\le \mathcal{C}(k)\|G\|_{Lip}\|H\|_\infty\qquad\forall k\in\mathbb{N},
$$
for functions $H$ which are constant on local stable leaves $W_s$ of $T$ and
the functions $G:M\to\mathbb{R}$ being Lipschitz continuous. This is ensured by Theorem 2.5 in \cite{DDHH}, where the role of $H$ is taken by the test functions in $C^{\kappa}(W, \mathbb{C})$ and $G\in \mathcal{B},$ which is the completion of Lipschitz functions  on $X.$ The decay is exponential. \\

{\em Cylinder sets.} The proof requires the existence, for each $n\ge 1$, of a partition of each unstable leaf in subsets $\xi^{(k)}_n$, called $n$-cylinders (or cylinders of rank $n$), and indexed with $k$,  where $T^n$ is defined and the image $T^n\xi^{(k)}_n$ is an unstable leaf of full length   for each $k$. These cylinders are obtained by taking the $2^n$ preimages of $\Gamma=\{y=\alpha\} $ by the map $T_Y$ restricted to each leaf. In the following, we will take $\alpha=1/2$ to simplify the exposition. \\

{\em Exact dimensionality of the SRB measure.} This uses the existence of the limit (\ref{ED}).
 We shall need the following result.

\begin{lemma}\label{annulus.condition}
({\em Annulus type  condition}) Let $w>1$. If $x$ is a point for which the dimension limit~\eqref{ED}
exists for a positive $d$, then there exists a $\delta>0$ so that
$$
\frac{\mu(B(x,r+r^w )\setminus B(x,r))}{\mu(B(x, r))}=
O(r^\delta),
$$
for all $r>0$ small enough.
\end{lemma}

Now we can apply the results of  Section 7.4 in~\cite{HHVV} to  prove the following result which tracks the number of visits a trajectory of the point $x \in X$ makes to the set $U$
on a suitable normalized orbit   segment:

\begin{proposition}\label{limiting.distribution}
 Consider  the counting function

\begin{equation*}
N^{\tau}_{B_n}(z)=\sum_{i=0}^{\floor{\tau/\mu(B_n)}} {{\bf 1}_{B_n}} \circ T^i(x),
\end{equation*}
 where $\tau$ is a positive
parameter and
$z$ is a point for which the limit (\ref{ED}) exists and $n\mu(B(z, e^{-u_n}))\rightarrow \tau$.
\begin{itemize}
\item If $z$ is not a periodic point and using the Euclidean metric, then we get a pure Poisson distribution:
$$\mu(N^{\tau}_{B_n}=k)\rightarrow \frac{e^{-\tau}\tau^k}{k!}, \ n\rightarrow \infty.$$
\item If $z$ is  a periodic point of minimal period $p$ and using the $l^{\infty}$ metric, we get a compound Poisson distribution
(P\'olya-Aeppli):
    $$
    \mu(N^{\tau}_{B_n}=k)\rightarrow  e^{-\Theta \tau}
\sum_{j=1}^k(1-\Theta)^{k-j}\Theta^{2j}\frac{\tau^j}{j!}\binom{k-1}{j-1}, \ n\rightarrow \infty,
    $$
\end{itemize}
where $\Theta$ is given  by $\Theta=1-\lim_{n\rightarrow \infty}\frac{\mu(T^{-p}B_n\cap B_n)}{\mu(B_n)}.$
\end{proposition}

\begin{proof}[Proof of Lemma~\ref{annulus.condition}.]
We have to prove the lemma in the two cases when (I)~the norm is $\ell^2$ and  
(II)~the norm is $\ell^\infty$ and the ball is geometrically a square.

\noindent (I) We now use the Euclidean metric and denote with $\mathcal{A}$
the annulus $\mathcal{A}=B(x,r+r^w )\setminus B(x, r)$ where $w>1.$ By disintegrating
 the SRB measure along the unstable manifolds we have:
$$
\mu(\mathcal{A})=\int l_{\nu}(\mathcal{A}\cap W_{\nu})\,
d\zeta(\nu).
$$
We  now split the subsets on each unstable manifold on the cylinders of rank
 $n$ and condition with respect to the Lebesgue measure on them\footnote{We use simply here $\xi_n$ instead of $\xi_n^{(k)}$ since the computation over $k$ is replaced by the sum.}:
\begin{equation}\label{FE}
l_{\nu}(\mathcal{A}\cap W_{\nu})=
\sum_{\xi_n; \xi_n\cap \mathcal{A}
\neq \emptyset}\frac{l_{\nu}(\mathcal{A}
\cap W_{\nu}\cap \xi_n)}{l_{\nu}(\xi_n)}l_{\nu}(\xi_n).
\end{equation}

 We then iterate forward each cylinder with $T^n;$
 they will become of full length equal to $1$ and subsequently we get $l_\nu'(T^n\xi_n)=1$  for some $W_{\nu'}.$
 Since the action  of $T$ is locally linear and expanding  by a factor $2^n$
(with the given choice of $\alpha=\frac12$) on the unstable leaves
and therefore has zero distortion, we have
 $$
 \frac{l_{\nu}(\mathcal{A}\cap W_{\nu}\cap \xi_n)}{l_{\nu}(\xi_n)}
=\frac{l_{\nu'}(T^n(\mathcal{A}\cap W_{\nu}\cap \xi_n))}{l_{\nu'}(T^n\xi_n)}
=l_{\nu'}(T^n(\mathcal{A})\cap W_{\nu'}),
 $$
 so that $T^n(\mathcal{A}\cap W_{\nu}\cap \xi_n)\subset W_{\nu'}$.
Therefore,
$$
l_{\nu}(\mathcal{A}\cap W_{\nu})
=\sum_{\xi_n; \xi_n\cap \mathcal{A}\neq \emptyset}l_{\nu'}
(T^n(\mathcal{A}\cap W_{\nu}\cap \xi_n))l_{\nu}(\xi_n).
$$

By elementary geometry we see that the largest intersection of $\mathcal{A}$ with the unstable
 leaves will produce a piece of length $O(r^{\frac{w+1}{2}});$ therefore $l_{\nu'}(T^n(\mathcal{A}\cap W_{\nu}\cap \xi_n))=O(2^n r^{\frac{w+1}{2}}),$
  and:
$$
\mu(\mathcal{A})=O(2^n r^{\frac{w+1}{2}})\int \sum_{\xi_n; \xi_n\cap \mathcal{A}
 \neq \emptyset} l_{\nu}(\xi_n)\,d\zeta(\nu).
$$
 We now observe that in order to have our result, it will be enough to get
 it with a decreasing sequence $r_n$, $n\rightarrow \infty$, of exponential type,
  $r_n=b^{-t(n)}, b>1, $ and $t(n)$ increasing to infinity. We put $r_n=2^{-n}.$ With this choice and remembering that $2^{-n}$ is also the length of the $n$-cylinders, we have
$$
\bigcup_{\xi_n; \xi_n\cap \mathcal{A} \neq \emptyset}\xi_n\subset B(x, r_n+r_n^w+2^{-n}) =B(x,2r_n+r_n^w)\subset B(x,3r_n),
$$
which, as the cylinders $\xi_n$ are disjoint, yields the estimate for the integral above:
$$
\mu(\mathcal{A})
=O( 2^n r_n^{\frac{w+1}{2}} r_n^{d-\epsilon}).
$$

 Now by the exact dimensionality of the SRB measure one has for any $\varepsilon>0$ and by renaming $r_n$ as $r:$
$$
(2r+r^w)^{d+\varepsilon}\le \mu(B(x, 2r+r^w))\le (2r+r^w)^{d-\varepsilon}
$$
for all $r$ small enough i.e.\ $n$ large enough.
With this we can divide $\mu(\mathcal{A})$ by the measure of the ball of radius $r$
and obtain the estimate
$$
\frac{\mu(\mathcal{A})}{\mu(B(x,r))}=O(r^{\frac{w-1}{2}+d-\varepsilon-d-\varepsilon})=
O(r^{\frac{w-1}{2}-2\varepsilon})=O(r^{\frac{w-1}{4}}),
$$
since $w>1,$ and provided $\varepsilon$ is small enough.

\noindent (II) Now we shall use the $\ell^\infty$-distance and again denote by $\mathcal{A}$
the annulus $B(x,r+r^w)\setminus B(x,r)$. Since we are in two dimensions, we can cover
the annulus by balls $B(y_j,2r^w)$ of radii $2r^w$, with centers $y_j$ for $j=1,\dots, N$.
The number $N$ of balls needed is bounded by $8\frac{r}{r^w}$.
For any $\varepsilon>0$ there exists a constant $c_1$ so that $\mu(B(y_j,2r^w))\le c_1r^{w(d-\varepsilon)}$
for all $r$ small enough. Thus
$$
\mu(\mathcal{A})
\le 8c_1r^{1+w(d-1-\varepsilon)}
$$
and since $\mu(B(x,r))\ge c_3r^{d+\varepsilon}$ for some $c_3>0$ we obtain
$$
\frac{\mu(\mathcal{A})}{\mu(B(x,r))}
\le c_4r^{(d-1)(w-1)-\varepsilon(w+1)}.
$$
The exponent $\delta=(d-1)(w-1)-\varepsilon(w+1)$ is positive as $d,w>1$ and $\varepsilon>0$
can be chosen sufficiently small.
\end{proof}

\begin{proof}[Proof of Proposition~\ref{limiting.distribution}] We can now prove the proposition by
applying Theorem~1 from~\cite{HHVV} and verify its Assumptions (I) to (VI) as follows:\\
(I) Let $\mathcal{I}_n$ be the collection of inverse branches $\varphi$ of the $n$-th iterate $T^n$
if the map $T$. Then, evidently, if $\varphi,\varphi'\in \mathcal{I}_n$ are two distinct inverse 
branches, their intersection $\varphi(X)\cap \varphi'(X)$ has zero measure. Therefore
the number of overlaps of `$n$-cylinders' $\varphi(X)$ is bounded by $L=1$.\\
(II) This condition is easily satisfied since  decay of correlations is exponential as 
the transfer operator is quasi compact.\\
(III) The set $\mathcal{G}_n$ of uniform expansion covers the entire space $X$ as
there is no `bad' set $\mathcal{G}_n^c$ of non-uniformly contracting inverse branches.
Consequently we get exponential contraction of $\max_{\varphi\in\mathcal{I}_n}\mbox{diam} \varphi(X)$
of the $n$-cylinder sets $\varphi(X)$. Moreover distortion 
$\sup_{\varphi\in\mathcal{I}}\sup_{x,y\in\varphi(X)}\frac{J_n(x)}{J_n(y)}$ 
is uniformly bounded, where $J_n$ is the Jacobian of $T^n$ restricted to the unstable
direction.\\
(IV) The dimension of the invariant measure is equal to $d=1+d_s$,
where $d_s<1$ is given above. So we can choose  $d_0>0$  and $d_1<\infty$ so that
$d_0<d<d_1$. \\
(V)  The dimension of the restricted measure on the unstable leaves equals
$u_0=1$ as it is Lebesgue.\\
(VI) This condition was verified in Lemma~\ref{annulus.condition}.

This shows that the condition of Theorem~1 of~\cite{HHVV} is satisfied.

If $x$ is an aperiodic point then  $\min\{j\ge1: B_\rho(x)\cap T^jB_\rho(x)\not=\varnothing\}$
goes to infinity as $\rho=e^{-u_n}\to0$. Thus for the coefficients
$$
\lambda_\ell(L)=\lim_{\rho\to0}\frac{\mathbb{P}(Z^L=\ell)}{\mathbb{P}(Z^L\ge1)}
$$
we obtain that for every $L$: $\lambda_1=1$ and $\lambda_\ell=0$ for all $\ell=2,3,\dots$,
where $Z^L=\sum_{j=1}^L{\bf 1}_{B_\rho(x)}\circ T^j$ is the hit counter on the finite orbit segment of length $L$.
This implies that $N_{B_n}^\tau$ converges in distribution to a standard Poisson
random variable with parameter $\tau$.

Let $x$ be a periodic point with minimal period $p$ and let $\tilde{B}_\rho$ be a square of
size $\rho$ centered at $x$ and whose sides are aligned with the stable and unstable directions
respectively. Then for $\ell=2,3,\dots$
$$
\hat\alpha_\ell
=\lim_{L\to\infty}\lim_{\rho\to0}\mathbb{P}(\tilde{Z}^L\ge\ell|\tilde{B}_\rho)
=\lim_{\rho\to0}\frac{\mu(\tilde{B}_\rho\cap T^{-(\ell-1)p}\tilde{B}_\rho)}{\mu(\tilde{B}_\rho)}
=\left(\lim_{\rho\to0}\frac{\mu(\tilde{B}_\rho\cap T^{-p}\tilde{B}_\rho)}{\mu(\tilde{B}_\rho)}\right)^{\ell-1}
$$
which implies that $\hat\alpha_\ell=\hat\alpha_2^{\ell-1}$, where $\tilde{Z}^L=\sum_{j=1}^L{\bf 1}_{\tilde{B}_\rho(x)}\circ T^j$.
Then for  $\alpha_\ell=\hat\alpha_\ell-\hat\alpha_{\ell+1}$
we thus obtain by~\cite{HHVV}  that
$\lambda_\ell=\frac{\alpha_\ell-\alpha_{\ell+1}}{\alpha_1}=(1-\theta)\Theta^{\ell-1}$,
where $1-\Theta=\alpha_1=1-\hat\alpha_2$ is the extremal index.
Hence $N_{B_n}^\tau$ converges in distribution to a P\'olya-Aeppli distributed random variable.
\end{proof}

It is worth mentioning that the previous result gives also an alternative way to prove EVT for the baker's map which is recovered when $k=0,$ as the limiting distribution of $\mu( N^{\tau}_{B_n}=0).$ Let us state it explicitly
\begin{proposition}\label{ty2}
Let $T$ be the baker's transformation and consider the function $M_n (x):=\max\{\phi(x), \dots, \phi(T^{n-1}x)\},$ where $\phi(x)=-\log d_{\infty}(x,z),$ and $z$ is chosen $\mu$-almost everywhere with respect to the SRB measure $\mu.$  Then, if $z$ is a  periodic point of minimal period $p$ verifying (\ref{ED}), we have
$$
\lim_{n\rightarrow \infty}\mu(M_n\le u_n)=e^{-\theta \tau},
$$
where $n\mu(B(z, e^{-u_n}))\rightarrow \tau$   and
$$
\theta=1-\alpha^p.
$$
\end{proposition}

\begin{remark}
Propositions \ref{ty1} and \ref{ty2} show that for a typical (non-periodic) point $z$ the limiting distribution of the maximum is purely exponential. The baker's map is probably the easiest example of a singular attractor. It is annoying that we could not compute analytically  the extremal index with respect to the Euclidean metric, which is the metric  usually accessible in simulations and physical observations. With references to Figures 2 and 3 respectively, the area of an extremely thin green ellipse within the blue circle is asymptotically equivalent to the area of an extremely thin green rectangle within the blue square, so, taking into account the ratio between the areas of the blue circle and the blue square, the limit as $p\rightarrow \infty$ of the extremal index for the Euclidean holes can be calculated.\footnote{We thank the anonymous referee for this observation.}
\end{remark}

\begin{example}\label{IIEE}
    \noindent The second statement of Proposition~\ref{limiting.distribution} about
periodic points requires the neighborhoods $B_n$ to be chosen in a dynamically
 relevant way. Here they turn out to be squares (or rectangles). If the measure
 has some mixing properties with respect to a partition then the sets $B_n$ can
 be taken to be cylinder sets as it was done in~\cite{HV09} for periodic points
 and in~\cite{HP14} Corollary~1 for non-periodic points. Here we show that for Euclidean balls
  one cannot in general expect the limiting distribution at periodic
  points to be P\'olya-Aeppli and therefore cannot be described by the single
  value of the extremal index.
 
  We assume that all parameters are equal, that
  is $\gamma_a=\gamma_b=\alpha=\beta=\frac12$. This is the {\em fat}
  baker's map for which the Lebesgue measure on $[0,1]^2$ is the SRB measure $\mu$.
  Let $x$ be a periodic point with minimal period $p$.
  Then $\mu(B(x,r))=r^2\pi$ and
  $$
  \mu\!\left(\bigcap_{i=0}^kT^{-ip}B(x,r)\right)=4r^22^{-kp}(1+\mathcal{O}(2^{-2kp})).
  $$
  This yields  
  $$
  \hat\alpha_{k+1}
  =\lim_{r\to0}\frac{\mu\!\left(\bigcap_{i=0}^kT^{-ip}B(x,r)\right)}{\mu(B(x,r))}
  =\frac4\pi \arctan 2^{-kp}
  =\frac4\pi2^{-kp}(1+\mathcal{O}(2^{-2kp}))
  $$
  for $k=1,2,\dots$. According to~\cite{HHVV} Theorem~2 we then define the values
   $\alpha_k=\hat\alpha_k-\hat\alpha_{k+1}$ where the value $\alpha_1$
   is the extremal index, i.e.\ $\theta=\alpha_1$. If the limiting distribution is
   P\'olya-Aeppli  then the probabilities $\lambda_k=\frac{\alpha_k-\alpha_{k+1}}{\alpha_1}$,
   $k=1,2,\dots$, are geometrically distributed and must satisfy $\lambda_k=\theta(1-\theta)^{k-1}$
   which is equivalent to saying that $\hat\alpha_{k+1}=(1-\theta)^k$ for $k=0,1,2,\dots$
   (see~\cite{HHVV} Theorem~2).
   Evidently this condition is violated in the present
    case and we conclude that
   the limiting distribution given by the values $\hat\alpha_k$ is not P\'olya-Aeppli
   and in fact obeys another compound Poisson distribution.
\end{example}
\subsection{Compound  point processes}
The compound Poisson distribution could be enriched by defining the {\em rare event point process} (REPP). Let us first introduce a few  objects. Put $I_l=[a_l, b_l), l=1,\dots,k, a_l, b_l\in \mathbb{R}_0^{+}$ a finite  number of disjoint semi-open intervals of the non-negative real axis; call $J=\cup_{l=1}^k I_l$ their  union. If $r$ is a positive real number, we write $rJ=\cup_{l=1}^k rI_l=\cup_{l=1}^k[ra_l, rb_l).$ We denote with $|I_l|$ the length of the interval $I_l,$ which coincides  with its Lebesgue measure $\text{Leb}(I_l).$ The REPP counts the number of visits to the set $B_n$ during the rescaled time period $v_n J:$
\begin{equation}\label{pp}
N_n(\cdot)(J)=\sum_{l\in v_nJ\cap\mathbb{N}_0}{\bf 1}_{B_n}(T^l\cdot),
\end{equation}
where $v_n$ is taken as
$$
v_n=\left\lfloor{\frac{\tau}{\mu(B_n)}}\right\rfloor, \ \tau>0.
$$
Our REPP belongs to the class of the point processes on $\mathbb{R}^+_0,$  see \cite{Kal} for all the properties of point processes used below. They are given by any measurable map $N:(X, \mathcal{F}_X, \mu)\rightarrow \mathcal{N}_p([0, \infty)),$  where $(X, \mathcal{F}_X, \mu)$ is the probability space of our original dynamical system with  the invariant measure $\mu$ and the Borel $\sigma$-algebra $\mathcal{F}_X,$ and $\mathcal{N}_p([0, \infty))$ denotes the set of counting measures $\textfrak{c}$ on $\mathbb{R}^+_0$ endowed with the  $\sigma$-algebra $\mathcal{M}_p(\mathbb{R}^+_0),$  which is the smallest $\sigma$-algebra making all evaluation maps $\textfrak{c}\rightarrow \textfrak{c}(B),$ from $\mathcal{N}_p([0, \infty))\rightarrow [0, \infty]$ measurable for all  $B\in \mathbb{R}_0^+.$ Any counting measure $\textfrak{c}$ has the form $\textfrak{c}=\sum_{i=1}^{\infty}\delta_{x_i}, \ x_i\in [0, \infty).$
The distribution  of $N$, denoted $\mu_N,$ is the measure $(\mu\circ N^{-1})=(\mu[N\in \cdot]),$ on $ \mathcal{M}_p(\mathbb{R}^+_0).$
The set  $\mathcal{N}_p([0, \infty))$ becomes a topological space with the vague topology, i.e. the sequence $\textfrak{c}_n$ converges to $\textfrak{c}$  whenever $\textfrak{c}_n(\phi)\rightarrow \textfrak{c}(\phi)$ for any continuous function $\phi:\mathbb{R}^+_0\rightarrow \mathbb{R}$ with compact support. We also say that the sequence of point processes $N_n$ converges in distribution to the point process $N,$ eventually defined on another probability space $(X', \mathcal{F}'_{X'}, \mu'),$ if $\mu_{N_n}$ converges weakly to $\mu'_N,$ that is for every continuous function $\varphi$ defined on $\mathcal{N}_p([0, \infty))$ we have $\lim_{n\rightarrow \infty}\int \varphi d\mu\circ N_n^{-1}=\int \varphi d\mu'\circ N^{-1}.$ In this case we will write $N_n\xrightarrow {\mu} N.$\\

If we now return to our REPP (\ref{pp}), we will see that a very common result is to get  
$
N_n\xrightarrow {\mu} \tilde{N},$
where
\begin{equation}\label{i3}
\mu(x, \tilde{N}(x)(I_l)=k_l, 1\le l\le n)=\prod_{l=1}^n e^{-\tau\text{Leb}(I_l)}\frac{\tau^{k_l}\text{Leb}(I_l)^{k_l}}{k_l!},
\end{equation}
for any disjoint bounded sets $I_1, \dots, I_n$ and non-negative integers $k_1,\dots, k_n,$ which is called the {\em standard Poisson point process}.  In general our REPP processes converges in distribution to a {\em compound point process} (CPP).
We say that  the point process  $N:(X', \mathcal{F}'_{X'}, \mu')\rightarrow \mathcal{N}_p([0, \infty))$ is a CPP with intensity parameter $t$ and cluster size distribution $(\lambda_l)_{l\ge 1}$ if  it satisfies:
\begin{itemize}
 \item For any finite sequence of measurable sets $B_1, \dots, B_k$ in $\mathcal{F}'_{X'}$ and mutually disjoint, the random variables $N(\cdot)(B_i), i=1,\dots,k,$ are independent.
 \item For any measurable set $B\in \mathcal{F}'_{X'},$ the random variable $N(\cdot)(B)$ is a CP random variable with intensity $t\text{Leb}(B), t>0$ and cluster size distribution $(\rho_l)_{l\ge 1},$ see the definition in section \ref{PPP}.
\end{itemize}

From now on we will simply write $N(\cdot)$ instead of $N(x)(\cdot)$ and we consider it as a CPP. In order to study the convergence of our REPP $N_n$ to the CPP $N$ two equivalent criteria are available. Before stating them we should  remind the definition of the
Laplace transform for a general point process $R:(X', \mathcal{F}'_{M'}, \mu')\rightarrow \mathcal{N}_p([0, \infty)):$
\begin{equation}\label{LT}
\psi_R(y_1, \dots, y_k)=\mathbb{E}_{\mu'}\left(e^{-\sum_{l=1}^{k}y_lR(I_l)}\right),
\end{equation}
for every non negative values $y_1, \dots, y_{k},$ each choice of $k$ disjoint intervals $I_i=[a_i, b_i), i=1,\dots,k.$ In the case of a  CPP $N$  with intensity parameter $t$ and cluster size distribution $(\rho_l)_{l\ge 1},$ we get
\begin{equation}\label{cs1}
\psi_{N}(y_1, \dots, y_{k})=e^{-t\sum_{l=1}^{k}(1-\varphi(y_l))\text{Leb}(I_l)},
\end{equation}
where $\varphi(y)=\sum_{i=0}^{\infty}e^{-yi}\rho_i$ is the  Laplace transform of the cluster size distribution $(\rho_l)_{l\ge 1}.$\\
 Therefore in order to establish the convergence in distribution of the REPP $N_n$ toward the CPP $N$ it will be sufficient to show \cite{Kal}: \\
- (C1):  that for any $k$ disjoint intervals $I_i=[a_i, b_i), i=1,\dots,k$ the joint distribution of $N_n$ converges to the joint distribution of $N,$ namely
$$
\left(N_n(I_1),\dots N_n(I_k)\right)\rightarrow \left(N(I_1),\dots N(I_k)\right).
$$
-(C2):  the convergence of the Laplace transforms:
$$
\psi_{N_n}(y_1, \dots, y_{\zeta})=\mathbb{E}\left(e^{-\sum_{l=1}^{k}y_lN_n(I_l)}\right)\rightarrow \psi_{N}(y_1, \dots, y_{k})=e^{-t\sum_{l=1}^{k}(1-\varphi(y_l))\text{Leb}(I_l)},
$$
as $n\rightarrow \infty.$\\
The criterion (C1) lends itself to being studied with  the probabilistic  approach of \cite{HHVV} as two of us recently  showed in  (\cite{amorim}, Theorem 3),  see also \cite{freitastodd} for a different method. The criterion (C2) is {\em naturally}  adapted to the spectral approach (just replacing characteristic functions with Laplace transforms), and the complete  treatment, involving two of us, will appear soon \cite{nuovo}. Both criteria allow to extend immediately Proposition \ref{limiting.distribution} to the point process framework giving
\begin{proposition}
 Consider  the counting measure
\begin{equation*}
N_n(\cdot)(J)=\sum_{l\in v_nJ\cap\mathbb{N}_0}{\bf 1}_{B_n}(T^l\cdot)
\end{equation*}
 where $\tau$ is a positive
parameter, $v_n=\left\lfloor{\frac{\tau}{\mu(B_n)}}\right\rfloor$,  and
$z$ is a point for which the limit (\ref{ED}) exists and $n\mu(B(z, e^{-u_n}))\rightarrow \tau$.
\begin{itemize}
\item If $z$ is not a periodic point and using the Euclidean metric, then $N_n$ converges in distribution to a standard Poisson point process of intensity $\tau,$ see (\ref{i3}) for the finite size distributions.
\item If $z$ is  a periodic point of minimal period $p$ and using the $l^{\infty}$ metric, we get a compound point process of
P\'olya-Aeppli type, namely a CPP with intensity $\tau \theta$ and cluster size distribution $\theta(1-\theta)^l, l\ge1,$
where $\theta$ is given as above by $\theta=1-\lim_{n\rightarrow \infty}\frac{\mu(T^{-p}B_n\cap B_n)}{\mu(B_n)}.$
\end{itemize}
\end{proposition}
\section{Generalization to other observable and connection with hitting time}\label{dse}
One could possibly wonder if the observable (\ref{OO}), $\Xi(x)=-\log d(x,z),$ plays an essential role in the theory. The answer is more nuanced. Let us consider a measurable function $\Phi:X\rightarrow \mathbb{R}\cup \pm \{\infty\}$ and construct the new process $\Phi\circ T^j, j\ge 0.$ We are interested in the extreme value distribution (EVD):
$$
\mathcal{W}_n=\mu(\mathcal{M}_n\le \textfrak{u}_n),
$$
where
$$
\mathcal{M}_n(x):=\max_{0\le j\le n-1}\{\Phi(T^jx)\}.
$$
We will return in a moment on the choice for sequence $\textfrak{u}_n.$ Let us introduce the set
$$
\textfrak{B}_n:=\{\Phi\ge  \textfrak{u}_n\},
$$
which we continue to call a {\em ball.} For instance, another commonly used observable is
$$
\Phi(x)=x^{-\frac{1}{\alpha}}, \ \alpha>0;
$$
in this case $\textfrak{B}_n$ is simply a closed euclidean ball  around zero of radius $\textfrak{u}_n^{-\alpha};$ see the Appendix A for a brief account of EVD for different types of observables.

Define now the quantity for any $x\in X$:
$$
\textfrak{t}_{\textfrak{B}_n}(x):=\inf\{j\ge 1; \Phi(T^jx)\in \textfrak{B}_n\},
$$
which gives the first {\em hitting time} to the ball $\textfrak{B}_n$ when we start from the point $x.$ By the invariance of the SRB measure $\mu$ it is easy to show that
\begin{equation}\label{retour}
\mu(\textfrak{t}_{\textfrak{B}_n}>n)=\mathcal{W}_n=\mu(\mathcal{M}_n\le \textfrak{u}_n),
\end{equation}
which establishes an important link between the law of extremes and the distribution of the  hitting times (this also enlightens again why the EVD is recovered when $k=0$ in the Poisson distribution, see Proposition \ref{ty2}).\\

Given the general observable $\Phi,$ the spectral analysis of section 3 proceeds formally as we did in the previous chapters, but in order to get the final results we need to check the following points:\\
(i) ${\bf 1}_{\textfrak{B}_n}\in \mathcal{B},$ and we saw that the geometrical shape of $\textfrak{B}_n$  matters.\\
(ii) Condition {\bf A2} must be checked taking into account again the geometrical nature of $\textfrak{B}_n.$\\
(iii) In dealing with condition {\bf A4} we  have now to compare
$\sup_{||h||\le 1}|Z(\mcl(h {\bf 1}_{\textfrak{B}_n}))|,$ see eq. (\ref{ETA}), with $
\Delta_n=\mu(\textfrak{B}_n);
$
this and the associated Lemma \ref{bcv} use again the local structure of  the set $\textfrak{B}_n.$\\
(iv) Finally we have to prove the convergence of the $q_k,$ (\ref{QQ}), to get the extremal index. This last condition requires the important scaling
$$
n\mu(\textfrak{B}_n)\rightarrow \tau,
$$
for some positive $\tau,$ which fixes as well the choice of $\textfrak{u}_n.$

\appendix
\section{Observables and corresponding extreme value laws}
The main classical result of extreme value theory is  given in the next theorem due to Gnedenko \cite{GN}, Fisher and Tippett \cite{FT}. The theorem deals with a sequence of i.i.d. random variables and we denote again with $M_n$ the maximum over the first $n$ variables. 
\begin{thm}
    If $X_0, X_1,\dots$ is a sequence of i.i.d. random variables and there exists linear normalizing sequences $(a_n)_{n\in \mathbb{N}}$ and $(b_n)_{n\in \mathbb{N}},$ with $a_n>0$ for all $n,$ such that
    $$
    \mathbb{P}(a_n(M_n-b_n)\le y)\rightarrow G(y),
    $$
    where the convergence occurs at continuity points of $G,$ and $G$ is nondegenerate, then $G(y)=e^{-\tau(y)},$ where $\tau(y)$ is one of the following three types (for some $\beta, \gamma>0$):\\
    (1) $\tau_1(y)=e^{-y}, y\in \mathbb{R};$\\
    (2) $\tau_2(y)=y^{-\beta}, y>0;$\\
    (3) $\tau_3(y)=(y)^{\gamma}, y>0.$
    \end{thm}
    We now give conditions on the choice of the observable  to get  sufficient and necessary conditions in order to get a nondegenerate EVD  still in the i.i.d. setting. 
    For the reader's convenience, we quote section 4.2.1 of the book \cite{book}, for the choice of the function $\Phi$ introduced in section \ref{dse}. It has the form, for $x\in X:$
$$
\Phi(x)=g(\text{dist}(x,\zeta)),
$$
where $\zeta\in X$ is a chosen point and the function $g:[0, +\infty)\rightarrow \mathbb{R}\cup\{+\infty\}$ is such that $0$ is a global maximum ($g(0)$ may be $\infty$); $g$ is a strictly decreasing bijection $g:V\rightarrow W$ in a neighborhood $V$ of $0$, and has one of the following three types of behavior:
\begin{itemize}
\item Type $g_1:$ there exists some strictly positive function $h:W\rightarrow \mathbb{R}$ such that for all $y\in \mathbb{R}:$
$$
\lim_{s\rightarrow g_1(0)}\frac{g_1^{-1}(s+yh(s))}{g_1^{-1}(s)}=e^{-y}.
$$
\item Type $g_2:$ $g_2(0)=+\infty$ and there exists $\beta>0$ such that for all $y>0:$
$$
\lim_{s\rightarrow +\infty}\frac{g_2^{-1}(sy)}{g_2^{-1}(s)}=y^{-\beta}.
$$
\item Type $g_3:$ $g_3(0)=D<+\infty$ and there exists $\gamma>0$ such that for all $y>0:
$
$$
\lim_{s\rightarrow +\infty}\frac{g_3^{-1}(D-sy)}{g_3^{-1}(D-s)}=y^{\gamma}.
$$
    \end{itemize}
Examples of each one of the three types are as follows:
\begin{enumerate}
 \item $g_1(x)=-\log x;$ in this case $h=1.$  
 \item $g_2(x)=x^{-1/\alpha}$ for some $\alpha>0;$ in this case $\beta=\alpha.$
 \item $g_3(x)=D-x^{1/\alpha}$ for some $D\in \mathbb{R}$ and $\alpha>0;$ in this case $\gamma=\alpha.$
\end{enumerate}
Type 1 gives the {\em Gumbel} law, type 2 gives the {\em Fréchet} law and finally type 3 furnishes the {\em Weibull} law.\\

    A great amount of work has been done to extend such a result from the i.i.d. setting  first to stationary and then to non stationary  processes. Whenever the latter arise in the dynamical systems setting, we defer to the book \cite{book} for an exhaustive presentation of the results: the spectral approach used in this article is one of them. As a final remark, we notice that by expressing our  scaling sequence $u_n$ as $u_n=b_n+\frac{y}{a_n},$  we will recover one of the previous three distributions as a function of the parameter $y.$
\section{Acknowledgments}
S.V. thanks the Laboratoire International Associ\'e LIA LYSM, the
INdAM (Italy), the UMI-CNRS 3483, Laboratoire Fibonacci (Pisa) where this work has
been completed under a CNRS delegation and the {\em Centro de Giorgi} in Pisa for various
supports. SV thanks M. Demers, D.~Dragi\v{c}evi\'{c}  C. Gonzalez-Tokman  and Y. Nakano for enlightening discussions.

N.H. was supported by a Simons Foundation Collaboration Grant (ID 526571).

J.A. was supported by an ARC Discovery Project and thanks the Centro de Giorgi in Pisa and CIRM in Luminy for
their support and hospitality.\\We finally thank the anonymous referees whose very careful and accurate reading of the paper, helped us to considerably improve it by adding several new proofs and explications.



\begin{thebibliography}{10}
\bibitem{amorim} L. Amorim, N. Haydn, and S. Vaienti. Compound Poisson distributions for random dynamical systems
using probabilistic approximations, to appear in  Stochastic processes and their applications, arXiv: 2402.02759

\bibitem{quenched} J. Atnip, G. Froyland, C. Gonzalez-Tokman, S. Vaienti, Thermodynamic Formalism and Perturbation Formulae for Quenched Random Open Dynamical Systems, to appear in Dissertationes Mathematicae, https://arxiv.org/pdf/2307.00774.pdf


          \bibitem{compound} J. Atnip, G. Froyland, C. Gonzalez-Tokman, S. Vaienti, Compound Poisson statistics for dynamical systems via spectral perturbation, https://arxiv.org/pdf/2308.10798.pdf

\bibitem{nuovo} J. Atnip, S. Vaienti et al., in preparation

 \bibitem{v5} H. Aytac, J. Freitas, S. Vaienti: Laws of rare events for deterministic and random dynamical systems, Trans. Amer Math. Soc., 367 (2015), 8229--8278,  http://arxiv.org/pdf/1207.5188

     \bibitem{v1} Th. Caby, D. Faranda, G. Mantica, S. Vaienti, P. Yiou: Generalized dimensions, large deviations and the distribution of rare events, {\em PHYSICA D},  Vol. 40015,Article 132143, 15 pages, (2019).

     \bibitem{v2}  Th. Caby, D. Faranda, S. Vaienti, P. Yiou: On the computation of the extremal index for time series,  {\em Journal of Statistical Physics}, 179, 1666-1697, (2020).

 \bibitem{NF} M. Carvalho, A. C. M. Freitas, J. M. Freitas, M. Holland, and M. Nicol: Extremal dichotomy for uniformly hyperbolic systems. Dyn. Syst., 30, no. 4, 383--403, 2015

\bibitem{DDMM} M. Demers:
\newblock Escape rates and physical measures for the infinite horizon Lorentz gas with holes,
\newblock {\em Dynamical Systems: An International Journal} 28:3 (2013), 393--422

\bibitem{DDHH} M. Demers: A Gentle Introduction to Anisotropic Banach Spaces, {\em Chaos, Solitons and Fractals}, 116 (2018), 29-42.

\bibitem{DL} M.~F. Demers and C.~Liverani:
\newblock Stability of statistical properties in two-dimensional piecewise
  hyperbolic maps.
\newblock {\em Trans. Amer. Math. Soc.}, 360(9): 4777--4814, 2008.


\bibitem{DZ1}M. Demers, H. Zhang:
\newblock  Spectral analysis of the transfer operator for the Lorentz gas.
\newblock {\em  Journal of Modern Dynamical Systems}, 5:4, 665--709, (2011).



\bibitem{v3}D. Faranda, H, Ghoudi, P. Guiraud, S. Vaienti: Extreme Value Theory for synchronization of  Coupled Map Lattices, {\em  Nonlinearity},  31, 7, 3326--3358 (2018).

\bibitem{feller} W. Feller. An introduction to probability theory and its applications. Vol. I. John Wiley and Sons, Inc.,
New York-London-Sydney, third edition, 1968.

\bibitem{freitastodd}A.C.M. Freitas, J.M. Freitas, M. Todd, The compound Poisson limit ruling periodic extreme behaviour of non-uniformly hyperbolic dynamics, Comm. Math. Phys., 321, no. 2, 483-527, 2013.

\bibitem{FT} A. Fisher, L. Tippett, Limiting foms of the frequency distribution of the largest or smallest member of a sample, Mathematical Proceedings of the Cambridge Philosophical Society, 24 (2), 180-190, 1928.

\bibitem{GN} B. Gnedenko, Sur la distribution limite du terme maximum d'une série aléatoire, Annals of Mathematics, 2 (44), 423-453, 1943.

 \bibitem{v4} P. Giulietti, P. Koltai, S. Vaienti: Targets and holes, submitted, Proceedings of the AMS,  149, N. 8, p. 3293-3306.
 
 \bibitem{HP14} N Haydn and Y Psiloyenis: Return times distribution for Markov towers with
decay of correlations; Nonlinearity, 27(6), (2014), 1323--1349.


 
 \bibitem{HV09}  N Haydn, S Vaienti: The compound Poisson distribution and return times in dynamical systems;
  Prob.\ Th.\ \& Related Fields  144, (2009), 517--542.

\bibitem{HHVV} N Haydn, S Vaienti: Limiting Entry and Return Times Distributions fir Arbitrary Null Sets,
 Communication Mathematical Physics, 378, 149-184, (2020).

 \bibitem{Kal} O. Kallenberg
{\em Random measures}, Akademia-Verlag Berlin, (1986).

\bibitem{GK}G. Keller: Rare events, exponential hitting times and extremal indices via spectral perturbation, {\em Dyn.\ Syst.},
27 (2012) 11--27.

\bibitem{KL} G. Keller and C. Liverani: Rare events, escape rates and quasistationarity: some exact formulae, J. Stat.\
Phys.\ 135 (2009), 519--534.

\bibitem{KL2}G. Keller and C. Liverani: Stability of the spectrum for transfer operators. Annali della Scuola Normale
Superiore di Pisa-Classe di Scienze, 28(1):141–152, 1999. 7

\bibitem{book} Valerio Lucarini, Davide Faranda, Ana Cristina Moreira Freitas, Jorge Milhazes Freitas, Mark Holland, Tobias Kuna, Matthew Nicol, Mike Todd, Sandro Vaienti: {\it Extremes and Recurrence in Dynamical Systems},  Wiley Interscience, 2016, Pure and Applied Mathematics : A Wiley Series of Texts, Monographs and Tracts.

\bibitem{EO} E. Ott: {\it Chaos in Dynamical Systems}, 2nd ed. Cambridge University Press, 2002.


\bibitem{viana} M. Viana, {\it Stochastic dynamics of deterministic systems},
Brazillian Math. Colloquium 1997, IMPA.


\end{thebibliography}
\end{document}